\documentclass[11pt]{amsart}

\usepackage{amssymb}
\usepackage{bm}
\usepackage{graphicx}
\usepackage{enumerate}
\usepackage{psfrag}
\usepackage{color}
\usepackage{url}
\usepackage{listings}

\usepackage{graphicx}
\usepackage{calc}
\usepackage{enumerate}
\usepackage{algpseudocode}

\newcommand{\excise}[1]{}%{$\star$\textsc{#1}$\star$}

\newtheorem{thm}{Theorem}[section]
\newtheorem{lemma}[thm]{Lemma}

\newtheorem{cor}[thm]{Corollary}
\newtheorem{prop}[thm]{Proposition}

\theoremstyle{definition}
\newtheorem{example}[thm]{Example}
\newtheorem{remark}[thm]{Remark}
\newtheorem{defn}[thm]{Definition}

\newtheorem{alg}[thm]{Algorithm}

\numberwithin{equation}{section}

\newcommand{\ring}[1]{\ensuremath{\mathbb{#1}}}

\renewcommand\>{\rangle}
\newcommand\<{\langle}

\newcommand\NN{\ring{N}}

\newcommand\ZZ{\ring{Z}}

\newcommand\kk{\mathbb K}

\DeclareMathOperator\In{In} % In
\DeclareMathOperator\supp{supp} % support
\newcommand\dist{{\mathsf d}}

\newcommand\ff{\phantom{-}}

\begin{document}%%%%%%%%%%%%%%%%%%%%%%%%%%%%%%%%%%%%%%%%%%%%%%%%%%%%%%%%
%%%%%%%%%%%%%%%%%%%%%%%%%%%%%%%%%%%%%%%%%%%%%%%%%%%%%%%%%%%%%%%%%%%%%%%%

\mbox{}
%\vspace{-2ex}%-1.1743pt}
\title{On the computation of factorization invariants for affine semigroups}

\author[Garc\'ia-S\'anchez]{Pedro Garc\'ia-S\'anchez}
\address{Departamento de \'Algebra\\ Universidad de Granada\\18071 Granada, Espa\~na}
\email{pedro@ugr.es}

\author[O'Neill]{Christopher O'Neill}
\address{Mathematics\\University of California, Davis\\One Shields Ave\\Davis, CA 95616}
\email{coneill@math.ucdavis.edu}

\author[Webb]{Gautam Webb}
\address{Mathematics Department\\University of Oregon\\Eugene, OR 97403}
\email{gwebb@uoregon.edu}

\date{\today}

\begin{abstract}
We present several new algorithms for computing factorization invariant values over affine semigroups.  In particular, we give (i) the first known algorithm to compute the delta set of any affine semigroup, (ii) an improved method of computing the tame degree of an affine semigroup, and (iii) a dynamic algorithm to compute catenary degrees of affine semigroup elements.  Our algorithms rely on theoretical results from combinatorial commutative algebra involving Gr\"obner bases, Hilbert bases, and other standard techniques.  Implementation in the computer algebra system GAP is discussed.  
\end{abstract}

\maketitle

% \setcounter{tocdepth}{1}
% \tableofcontents

%%%%%%%%%%%%%%%%%%%%%%%%%%%%%%%%%%%%%%%%%%%%%%%%%%%%%%%%%%%%%%%%%%%%%%%%%
\section{Introduction} \label{sec:intro}%%%%%%%%%%%%%%%%%%%%%%%%%%%%%%%%%%%%
%raggedbottom%%%%%%%%%%%%%%%%%%%%%%%%%%%%%%%%%%%%%%%%%%%%%%%%%%%%%%%%%%%%

Let $\Gamma = \<\alpha_1, \ldots, \alpha_k\> \subset \NN^d$ be a finitely generated subsemigroup of $\NN^d$, that is, an \emph{affine semigroup}.  We are interested in studying the factorization structure of elements $\gamma \in \Gamma$, that is, the distinct expressions of $\gamma$ as a sum of the generators $\alpha_1, \ldots, \alpha_k$.  Most of the theory on nonunique factorization stems from the study of factorizations in integral domains; in this context, one can view factorizations of affine semigroup elements as factorizations of monomials in the corresponding monoid algebra $\kk[\Gamma]$ for $\kk$ a field.  Often, factorization structure is examined using factorization invariants, which assign quantities to each element of $\Gamma$ (or to $\Gamma$ as a whole) that measure the failure of its factorizations to be unique \cite{g-hk}.  
% Factorization invariants are often combinatorial in nature, and provide concrete methods of quantifying the abundance and variety of factorizations.  
% $\kk[\Gamma] = \kk[t_1^{\gamma_1} \cdots t_d^{\gamma_d}]$.  

Every affine semigroup $\Gamma$ is an \emph{FF-monoid}, meaning that each element $\gamma \in \Gamma$ has only finitely many factorizations.  As such, many factorization invariants can be computed explicitly using computer software \cite{compoverview}, an attribute that has greatly aided their study \cite{numericalsurvey}.  On the other hand, the class of affine semigroups is sufficiently broad to exhibit varied factorization structure, and remains an interesting source of examples and an active area of study.  

In this paper, we present new algorithms for computing several well-studied factorization invariants.  After some preliminary definitions in Section~\ref{sec:background}, we present in Section~\ref{sec:deltaset} the first known algorithm for computing the delta set of an affine semigroup.  Two distinct approaches are given (Algorithms~\ref{a:affinedeltahilbert} and~\ref{a:affinedeltagb}), the latter of which introduces a novel connection between the delta set and Gr\"obner bases of homogeneous toric ideals (see Remark~\ref{r:affinedeltagb}).  
Next, we give Algorithm~\ref{a:dynamiccatenary} for dynamically computing catenary degrees of affine semigroup elements, in the style of several known algorithms for other factorization invariants \cite{dynamicalg}.  
We conclude with an improved algorithm for computing the tame degree of an affine semigroup in Section~\ref{sec:tamedegree}.  

Following the statement and proof of correctness of each algorithm, we discuss implementation in the \texttt{GAP} \cite{gap} package \texttt{numericalsgps} \cite{numericalsgps}, including sample code and benchmark comparisons with existing algorithms.  Whenever a Gr\"obner basis (Definition~\ref{d:grobnerbasis}) or a Hilbert basis (Definition~\ref{d:hilbertbasis}) must be computed, our implementations rely on trusted software packages like \texttt{4ti2}~\cite{4ti2}, \texttt{Normaliz}~\cite{normaliz}, and \texttt{Singular}~\cite{singular}.

\subsection*{Notation}

In what follows, let $\NN = \{0, 1, 2, \ldots\}$ denote the set of non-negative integers.  Fix a field $\kk$, and let $\kk[y_1, \ldots, y_k]$ denote the polynomial ring in commuting variables $y_1, \ldots, y_k$ with coefficients in $\kk$.  Lastly, for $z \in \NN^k$ we use the shorthand $y^z = y_1^{z_1} \cdots y_k^{z_k} \in \kk[y_1, \ldots, y_k]$.  

%%%%%%%%%%%%%%%%%%%%%%%%%%%%%%%%%%%%%%%%%%%%%%%%%%%%%%%%%%%%%%%%%%%%%%%%%
\section{Background}%%%%%%%%%%%%%%%%%%%%%%%%%%%%%%%%%%%%%%%%%%%%%%%%%%%%%
\label{sec:background}%%%%%%%%%%%%%%%%%%%%%%%%%%%%%%%%%%%%%%%%%%%%%%%%%%%
%raggedbottom%%%%%%%%%%%%%%%%%%%%%%%%%%%%%%%%%%%%%%%%%%%%%%%%%%%%%%%%%%%%

\begin{defn}\label{d:affine}
A semigroup $\Gamma \subset \NN^d$ is \emph{affine} if it is finitely generated.  If $\Gamma \subset \NN$ and $\gcd(\Gamma) = 1$, we say $\Gamma$ is a \emph{numerical semigroup}.  
\end{defn}

\begin{remark}\label{r:minimalgenerators}
Any affine semigroup $\Gamma$ is reduced, and thus has a unique generating set that is minimal with respect to containment.  Throughout this paper, whenever we write $\Gamma = \<\alpha_1, \ldots, \alpha_k\>$, we assume $\alpha_1, \ldots, \alpha_k$ are precisely the minimal generators of $\Gamma$.  
% Every affine semigroup is an FF-monoid and thus an BF-monoid.
\end{remark}

\begin{defn}\label{d:hilbertbasis}
Fix a matrix $A \in \ZZ^{k \times d}$.  The set 
\[
\Gamma = \big\{x \in \NN^d : Ax = 0\big\}
\]
of nonnegative integer homogeneous solutions of $A$ forms an affine semigroup.  The minimal generating set of $\Gamma$ (w.r.t.\ containment) is called the \emph{Hilbert basis} of $A$, denoted $\mathcal H(\Gamma)$.  
% The set $\mathcal H(\Gamma)$ corresponds with $\mathrm{Minimals}_\le(\Gamma\setminus\{0\})$, that is, the set of nonzero minimal elements of $\Gamma$ with respect to the usual partial ordering on $\mathbb N^d$.
\end{defn}

\begin{defn}\label{d:factset}
Fix an affine semigroup $\Gamma = \<\alpha_1, \ldots, \alpha_k\> \subset \NN^d$.  The elements $\alpha_1, \ldots, \alpha_k$ comprising the unique minimal generating set of $\Gamma$ are called \emph{irreducible} elements (or \emph{atoms}).  A \emph{factorization} of $\gamma \in \Gamma$ is an expression 
$$\gamma = z_1\alpha_1 + \cdots + z_k\alpha_k$$
of $\gamma$ as a finite sum of atoms, which we denote by the $k$-tuple $z = (z_1, \ldots, z_k) \in \NN^k$.  Write $\mathsf Z_\Gamma(\gamma)$ for the \emph{set of factorizations} of $\gamma \in \Gamma$, viewed as a subset of $\NN^k$.  
\end{defn}

Since $\alpha_1,\ldots,\alpha_k$ generate $\Gamma$, the monoid homomorphism $\varphi_\Gamma : \mathbb N^k \to \Gamma$ given by
\[
\varphi_\Gamma(z_1, \ldots, z_k) = z_1\alpha_1 + \cdots + z_k\alpha_k
\]
is surjective, and $\Gamma \cong \mathbb N^k/\ker\varphi_\Gamma$, where 
\[
\ker \varphi_\Gamma =\{ (z,w) \in \mathbb N^k\times \mathbb N^k : \varphi_\Gamma(z) = \varphi_\Gamma(w)\}
\]
denotes the \emph{kernel} of $\varphi_\Gamma$.  Notice that $\ker\varphi_\Gamma$ is a \emph{congruence} on $\mathbb N^k$ (i.e.\ an equivalence relation that is closed under translation), and thus it is finitely generated (see \cite{redei}).  In particular, there exists a finite set $\rho\subseteq \ker\varphi_\Gamma$ such that the smallest congruence containing $\rho$ is $\ker\varphi_\Gamma$ (equivalently, $\ker\varphi_\Gamma$ equals the intersection of all congruences containing $\rho$). 

\begin{defn}\label{d:minimalpresentation}
A \emph{minimal presentation} of $\Gamma$ is a generating set $\rho$ of the congruence $\ker\varphi_\Gamma$ that is minimal with respect to containment.  The \emph{Betti elements} of $\Gamma$ are elements of the form $\varphi_\Gamma(z)$ for $(z,w) \in \rho$ (this is independent of the minimal presentation chosen).  
\end{defn}

%Let $M$ be an array with integer coefficients and $c$ columns. The set $F=\{x\in \mathbb N^c\colon Mx=0\}$ is an affine semigroup of $\mathbb N^c$. These semigroups are sometimes called \emph{full} affine semigroups, and they are generated by the set of minimal nonzero elements with respect to the usual partial ordering on $\mathbb N^c$. This set of generators is known as a \emph{Hilbert basis} of $M$, and we denote it by $\mathcal H(F)$. 

\begin{defn}\label{d:graverbasis}
Fix $\Gamma = \<\alpha_1, \ldots, \alpha_k\> \subset \NN^d$, and let $A = (\alpha_1 \mid \cdots \mid \alpha_k)$ denote the matrix with columns $\alpha_1, \ldots, \alpha_k$.  The \emph{Graver basis} of $A$ (or the \emph{Graver basis} of $\Gamma$) is the Hilbert basis of the matrix $(A \mid -A)$.  
% We define the Graver basis of $\Gamma$ as the Graver basis of $A$.
\end{defn}

\begin{remark}\label{r:graverminpres}
The Graver basis of a semigroup $\Gamma$ contains every minimal presentation of $\Gamma$, but is typically orders of magnitude larger.  Even for numerical semigroups with 3 minimal generators, whose minimal presentations have at most 3 relations, the Graver basis can be arbitrarily large.  As such, algorithms relying on a proper subset of the Graver basis (such as a minimal presentation or a Gr\"obner basis) are preferred whenever possible.  
\end{remark}

\begin{example}\label{e:minimalpresentation}
Let $\Gamma = \<3,4,5\> \subseteq \NN$.  The minimal presentation of $\Gamma$ is 
\[
\rho=\{((1,0,1),(0,2,0)),((2,1,0),(0,0,2)),((3,0,0),(0,1,1))\},
\]
as computed in \cite[Example 8.23]{numerical}.  In this case, $\rho$ is the only minimal presentation of $\Gamma$, though minimal presentations need not be unique in general.  For comparison, the Graver basis of $\Gamma$ consists of (upon omitting symmetry) the pairs
\[
\begin{array}{ll}
((1, 0, 1), (0, 2, 0)),
((1, 3, 0), (0, 0, 3)),
((2, 1, 0), (0, 0, 2)),\\
((3, 0, 0), (0, 1, 1)),
((4, 0, 0), (0, 3, 0)),
((5, 0, 0), (0, 0, 3)).
% full thing, courtesy of normaliz
% ((0, 0, 1), (0, 0, 1)),
% ((0, 0, 2), (2, 1, 0)),
% ((0, 0, 3), (1, 3, 0)),
% ((0, 0, 3), (5, 0, 0)),
% ((0, 0, 4), (0, 5, 0)),
% ((0, 1, 0), (0, 1, 0)),
% ((0, 1, 1), (3, 0, 0)),
% ((0, 2, 0), (1, 0, 1)),
% ((0, 3, 0), (4, 0, 0)),
% ((0, 5, 0), (0, 0, 4)),
% ((1, 0, 0), (1, 0, 0)),
% ((1, 0, 1), (0, 2, 0)),
% ((1, 3, 0), (0, 0, 3)),
% ((2, 1, 0), (0, 0, 2)),
% ((3, 0, 0), (0, 1, 1)),
% ((4, 0, 0), (0, 3, 0)),
% ((5, 0, 0), (0, 0, 3))
\end{array}
\]
\end{example}

We associate to $\Gamma$ an ideal $I_\Gamma \subset \kk[y_1, \ldots, y_k]$ defined as follows.  Consider the polynomial ring $\kk[t_1, \ldots, t_d]$, and more specifically the subring $\kk[\Gamma] = \bigoplus_{\gamma \in \Gamma} \mathbb{K} t^\gamma$.  In this setting, $\varphi_\Gamma$ has a polynomial ring counterpart, namely the ring homomorphism 
\[
\psi_\Gamma : \mathbb K[y_1,\ldots, y_k]\to \mathbb K[\Gamma]
\]
determined by $\psi_\Gamma(y_i) = t^{\alpha_i}$. The \emph{defining ideal} of $\Gamma$ is the kernel
\[
I_\Gamma = \ker\psi_\Gamma \subset \kk[y_1, \ldots, y_k].
\]
It is well known that a set $\rho \subset \ker\varphi_\Gamma$ generates $\ker\varphi_\Gamma$ as a congruence if and only if 
\[
\left\{y^z - y^w : (z,w) \in \rho\right\} \subset I_\Gamma
\]
generates $I_\Gamma$, and that $\rho$ is a minimal presentation for $\Gamma$ precisely when the above set irredundantly generates $I_\Gamma$ (see \cite{herzog}).

We conclude this section with a brief discussion of Gr\"obner bases, which are standard in computational commutative algebra; the unfamiliar reader is encouraged to consult~\cite{clo} for a more thorough overview.  

\begin{defn}\label{d:grobnerbasis}
Fix a monomial order $\preceq$ on $\kk[y_1, \ldots, y_k]$ (that is, a total ordering of the monomials containing the usual partial ordering by divisibility).  Given $f \in \kk[y_1, \ldots, y_k]$, denote by $\In_\preceq(f)$ the \emph{initial term} of $f$ with respect to $\preceq$.  A \emph{Gr\"obner basis} of an ideal $I \subset \kk[y_1, \ldots, y_k]$ with respect to $\preceq$ is a set $G = \{g_1, \ldots, g_r\}$ generating $I$ such that $\<\In_\preceq(g_1), \ldots, \In_\preceq(g_r)\> = \langle \{ \In_\preceq(f) \colon f\in I\} \rangle$.  
% We say $G$ is \emph{reduced} if $g_1,\ldots, g_r$ are monic and $\In_\preceq(g_i)$ divides no term of $g_j$ for $i \ne j$.
\end{defn}

% \begin{remark}\label{r:reducedgb}
% A reduced Gr\"obner basis is unique for fixed $I$ and $\preceq$ (see for instance~\cite{clo}).  
% \end{remark}

\begin{remark}\label{r:elimination}
Minimal presentations can be computed using Gr\"obner bases. Set 
\[
J = \<y_1 - t^{\alpha_1}, \ldots, y_k - t^{\alpha_k}\> \subseteq \kk[y_1, \ldots, y_k, t_1, \ldots, t_d],
\]
and observe that $I_\Gamma = J \cap \kk[y_1, \ldots, y_k]$.  To compute a generating set for $I_\Gamma$, one can compute a Gr\"obner basis $G$ for $J$ with respect to any monomial ordering $\preceq$ satisfying $y_i \preceq t_j$ for all possible $i$ and $j$, and then calculate $G \cap \kk[y_1,\ldots, y_k]$.
\end{remark}

%%%%%%%%%%%%%%%%%%%%%%%%%%%%%%%%%%%%%%%%%%%%%%%%%%%%%%%%%%%%%%%%%%%%%%%%%
\section{The delta set}%%%%%%%%%%%%%%%%%%%%%%%%%%%%%%%%%%%%%%%%%%%%%%%%%%
\label{sec:deltaset}%%%%%%%%%%%%%%%%%%%%%%%%%%%%%%%%%%%%%%%%%%%%%%%%%%%%%
%raggedbottom%%%%%%%%%%%%%%%%%%%%%%%%%%%%%%%%%%%%%%%%%%%%%%%%%%%%%%%%%%%%

In this section, we introduce Algorithms~\ref{a:affinedeltahilbert} and~\ref{a:affinedeltagb}, the first known algorithms for computing the delta set of any affine semigroup, as well as proofs of correctness (Theorems~\ref{t:affinedeltahilbert} and~\ref{t:affinedeltagb}, respectively).  Both algorithms have Theorem~\ref{t:deltahilbert} at their core, but each uses a different method for computing the necessary ideal generating sets.  
% Preceeding Algorithm~\ref{a:affinedeltahilbert} is a discussion of various design decisions made by the authors.  
% Following Theorem~\ref{t:affinedeltahilbert} are benchmarks and specific implementation details.  

\begin{defn}\label{d:deltaset}
Fix an affine semigroup $\Gamma = \<\alpha_1, \ldots, \alpha_k\> \subset \NN^d$ and $\gamma \in \Gamma$.  Given $z \in \mathsf Z_\Gamma(\gamma)$, the \emph{length} of $z$ is the number $|z| = z_1 + \cdots + z_k$ of irreducibles in $z$.  The \emph{length set} of $\gamma$ is
\[
\mathsf L_\Gamma(\gamma) = \{z_1 + z_2 + \cdots + z_k : z \in \mathsf Z_\Gamma(\gamma)\},
\]
the set of factorization lengths.  Writing $\mathsf L_\Gamma(\gamma) = \{\ell_1 < \cdots < \ell_m\}$, the \emph{delta set} of $\gamma$ is the set
$$\Delta(\gamma) = \{\ell_{i+1} - \ell_i : 1 \le i < m\}$$
of successive differences of factorization lengths.  The \emph{delta set of $\Gamma$} is $\Delta(\Gamma) = \bigcup_{\gamma \in \Gamma} \Delta(\gamma)$.  
% We say $\Gamma$ is \emph{half-factorial} if $|\mathsf L_\Gamma(\gamma)| = 1$ for all $\gamma \in \Gamma$.  
\end{defn}

Theorem~\ref{t:minmaxdeltaset} follows from \cite[Corollary 2.4, Thoerem 2.5]{max-delta}, and describes which delta set elements can be easily recovered from a minimal presentation.  

\begin{thm}\label{t:minmaxdeltaset}
If $\Gamma = \<\alpha_1, \ldots, \alpha_k\> \subset \NN^d$ and $\rho \subset (\NN^k)^2$ is a minimal presentation of $\Gamma$, then 
$$\min\Delta(\Gamma) = \gcd(\{|z| - |w| : (z,w) \in \rho\}),$$
$$\max\Delta(\Gamma) = \max\{\max\Delta(z_1\alpha_1 + \cdots + z_k\alpha_k) \mid (z,w) \in \rho\}$$
and 
\[
\Delta(\Gamma)\subseteq [\min\Delta(\Gamma),\max\Delta(\Gamma)]\cap \min\Delta(\Gamma)\mathbb Z.
\]
In particular, $\min\Delta(\Gamma) = \gcd\Delta(\Gamma)$, and both $\min\Delta(\Gamma)$ and $\max\Delta(\Gamma)$ can be quickly recovered from $\rho$.  
\end{thm}

\begin{example}\label{e:deltaset}
Let us return to $\Gamma = \<3,4,5\>$ from Example~\ref{e:minimalpresentation}.  
% We have $\mathsf Z_\Gamma(10) = \{(2,1,0),(0,0,2)\}$.  Consequently $\mathsf L_\Gamma(10)=\{2,3\}$ and $\mathsf \Delta(10)=\{1\}$.  
According to Theorem~\ref{t:minmaxdeltaset}, the minimum and maximum of $\Delta(\Gamma)$ are determined by $\Delta(8)$, $\Delta(9)$ and $\Delta(10)$.  The only factorizations of $8$ are $(0,2,0)$ and $(1,0,1)$, so $\Delta(8)=\emptyset$.  We can obtain $\Delta(9) = \Delta(10) = \{1\}$ similarly.  As such, $\Delta(\Gamma)=\{1\}$, since in this case $\min\Delta(\Gamma) = \max\Delta(\Gamma)$.  In general, however, Theorem~\ref{t:minmaxdeltaset} only yields an interval containing $\Delta(\Gamma)$.
\end{example}

\begin{remark}\label{r:deltadifficulty}
The primary difficulty in computing $\Delta(\Gamma)$ for a general affine semigroup $\Gamma$ is ensuring that a given value does \emph{not} occur in $\Delta(\Gamma)$.  
% without essentially knowing the delta set of every element of $\Gamma$.  
Indeed, some elements of $\Delta(\Gamma)$ may only occur in the delta sets of a small (finite) number of semigroup elements.  For example, if $\Gamma = \<17,33,53,71\> \subset \NN$, then $6 \in \Delta(\Gamma)$ only occurs in $\Delta(266)$, $\Delta(283)$, and $\Delta(300)$.  

As such, although it is computationally feasible to compute the delta set of any single element of $\Gamma$ (since each has only finitely many factorizations), this cannot be accomplished for all of the (infinitely many) elements of $\Gamma$.  To date, all existing delta set algorithms are limited to numerical semigroups, and act by bounding the semigroup elements for which ``new'' elements of $\Delta(\Gamma)$ can occur, thus restricting computation to a finite list of semigroup elements; see \cite{dynamicalg} for more detail.  

One of the primary benefits of Algorithm~\ref{a:affinedeltahilbert} is that it does not rely on computing delta sets of individual semigroup elements.  As such, both of our algorithms are generally much faster than existing algorithms for numerical semigroups; see Examples~\ref{e:affinedeltahilbert} and~\ref{e:affinedeltagb}.  
\end{remark}

We now state Theorem~\ref{t:deltahilbert}, the main theoretical result used in Algorithm~\ref{a:affinedeltahilbert}.  

\begin{thm}[{\cite[Theorem~4.8]{factorhilbert}}]\label{t:deltahilbert}
Suppose $\Gamma = \<\alpha_1, \ldots, \alpha_k\> \subset \NN^d$.  
The ideals
$$I_j = \left\<y^z - y^w : z, w \in \mathsf Z_\Gamma(\gamma), \gamma \in \Gamma, \text{ and } \big||w| - |z|\big| \le j\right\> \subset \kk[y_1, \ldots, y_k]$$
% a_1\alpha_1 + \cdots + a_r\alpha_r = b_1\alpha_1 + \cdots + b_r\alpha_r
for $j \ge 0$ form an ascending chain 
$$I_0 \subset I_1 \subset I_2 \subset \cdots$$
in which $I_{j-1} \subsetneq I_{j}$ if and only if $j \in \Delta(\Gamma)$.  
% and each $\gamma \in \Gamma$ satisfies $\mathcal H(I_j/I_{j-1};\gamma) > 0$ if and only if $j \in \Delta(\gamma)$.  
\end{thm}

Algorithm~\ref{a:affinedeltahilbert} obtains the delta set of a given affine semigroup $\Gamma = \<\alpha_1, \ldots, \alpha_k\> \subset \NN^d$ by computing each ideal defined in Theorem~\ref{t:deltahilbert}.  The generating set of $I_j$ appearing in Theorem~\ref{t:deltahilbert} is determined by (integer solutions to) a system of linear inequalities, and thus can be computed using a Hilbert basis.  For example, by inserting a slack variable to transform the single inequality to an equality, it suffices to compute a Hilbert basis for the matrix $A$ below.  Each element $x = (z, w, i) \in \NN^k \times \NN^k \times \NN$ of the Hilbert basis with $i \le j$ yields a binomial generator $y^z - y^w$ for $I_j$ by Lemma~\ref{l:affinedeltahilbert}.  

\begin{center}
$A = \left(\begin{array}{clclllr@{}c}
\alpha_{11} & \cdots & \alpha_{1k} & -\alpha_{11} & \cdots & -\alpha_{1k} & & 0 \\
\vdots & \ddots & \vdots & \ff \vdots & \ddots & \ff \vdots & & \vdots \\
\alpha_{d1} & \cdots & \alpha_{dk} & -\alpha_{d1} & \cdots & -\alpha_{dk} & & 0 \\
\multicolumn{1}{l}{1} & \cdots & \multicolumn{1}{l}{1} & -1 & \cdots & -1 & - & 1
\end{array}\right)$
\end{center}

Once each generating set has been computed, successive ideals $I_{j-1} \subset I_j$ can be checked for strict containment by computing a reduced Gr\"obner basis for each ideal (with respect to the same term order) and then comparing.  The final step in the algorithm is locating the maximal value of $j$, which by Theorem~\ref{t:minmaxdeltaset} can be obtained by computing a minimal presentation for $\Gamma$.  Note that each element of $\Delta(\Gamma)$ is a multiple of $\min\Delta(\Gamma)$, which by Theorem~\ref{t:minmaxdeltaset} can also be obtained from a minimal presentation for $\Gamma$.  

We are now ready to state Algorithm~\ref{a:affinedeltahilbert}.  

\begin{alg}\label{a:affinedeltahilbert}
Computes the delta set of an affine semigroup $\Gamma = \<\alpha_1, \ldots, \alpha_k\> \subset \NN^d$.  
\begin{algorithmic}
\Function{DeltaSetOfAffineSemigroup}{$\Gamma$}
\State $\rho \gets$ minimal presentation for $\Gamma$
\State $D \gets \bigcup_{(z,w) \in \rho} \Delta(z_1\alpha_1 + \cdots + z_k\alpha_k)$
\State $m \gets \gcd D$
\State $H \gets$ Hilbert basis for $\{x \in \NN^{2k+1} : Ax = 0\}$
\State $G \gets$ Gr\"obner basis for $\<y^z - y^w : (z,w,m) \in H \text{ or } (z,w,0) \in H\> \subset \kk[y_1, \ldots, y_k]$
\ForAll{$j = 2m, \ldots, \max D - m$}
	\If{$y^z - y^w$ has nonzero remainder modulo $G$ for some $(z,w,j) \in H$}
		\State $D \gets D \cup \{j\}$
		\State $G \gets$ Gr\"obner basis for $\<G\> + \<y^z - y^w : (z,w,j) \in H\>$
	\EndIf
	% \ForAll{$i = 1, 2, \ldots, k$ with $m - n_i \in S$}
	% 	\State $\mathsf L \gets \mathsf L \cup \{l + 1 : l \in \mathcal L[m - n_i]\}$
	% \EndFor
\EndFor
\State \Return $D$
\EndFunction
\end{algorithmic}
\end{alg}

We now give a proof that Algorithm~\ref{a:affinedeltahilbert} gives the correct output.  

\begin{lemma}\label{l:affinedeltahilbert}
Resume notation from Theorem~\ref{t:deltahilbert} and Algorithm~\ref{a:affinedeltahilbert}.  The ideal
$$J_j = \<y^z - y^w : (z,w,i) \in H, i \le j\> \subset \kk[y_1, \ldots, y_k]$$
computed in Algorithm~\ref{a:affinedeltahilbert} coincides with the ideal $I_j$ from Theorem~\ref{t:deltahilbert} for all $j \ge 0$.  
\end{lemma}

\begin{proof}
For each $i \le j$ and $(z, w, i) \in H$, it is clear that $y^z - y^w \in I_j$, so $J_j \subset I_j$.  Conversely, if $(z, w, i), (z', w', i') \in \<H\>$ with $y^z - y^w, y^{z'} - y^{w'} \in I_j$, then
$$y^{z'}(y^z - y^w) + y^w(y^{z'} - y^{w'}) = y^{z + z'} - y^{w + w'} \in I_j,$$
which is the binomial corresponding to $(z, w, i) + (z', w', i') \in \<H\> \subset \NN^{2k+1}$.  This means the generating set for $I_j$ given in Theorem~\ref{t:deltahilbert} can be restricted to the binomials corresponding to elements of the Hilbert basis $H$, so $I_j \subset J_j$ as well.  
\end{proof}

\begin{thm}\label{t:affinedeltahilbert}
For any affine semigroup $\Gamma \subset \NN^d$, Algorithm~\ref{a:affinedeltahilbert} returns $\Delta(\Gamma)$.  
\end{thm}

\begin{proof}
Resume notation from Theorem~\ref{t:deltahilbert} and Algorithm~\ref{a:affinedeltahilbert}.  Theorem~\ref{t:minmaxdeltaset} implies that $m = \min\Delta(\Gamma)$, and that only multiples of $m$ can appear in $\Delta(\Gamma)$.  Lemma~\ref{l:affinedeltahilbert} ensures each loop iteration computes the ideal $I_j$ for some multiple $j$ of $m$, and appends $j$ to $D$ if $I_{j-1} \subsetneq I_j$.  Applying Theorem~\ref{t:deltahilbert} completes the proof.  
\end{proof}

Algorithm~\ref{a:affinedeltagb} uses an alternative approach to compute generators for the ideals $I_j$ in Theorem~\ref{t:deltahilbert}.  
% We decided to search for an alternative characterization for the Delta sets of an affine semigroup.  
Given $\alpha\in \mathbb N^d$, denote by $(1,\alpha)$ the element in $\mathbb N^{d+1}$ whose first coordinate is~1 and whose remaining coordinates are those in $\alpha$.  For $\Gamma = \<\alpha_1, \ldots, \alpha_k\> \subset \NN^d$, set 
\[
\Gamma_H = \<(1,0),(1,\alpha_1), \ldots, (1,\alpha_k)\> \subset \NN^{d+1}.
\]
There is a tight relation between factorizations in $\Gamma_H$ and those in $\Gamma$ \cite{half-factorial}. Observe that 
\[\mathrm{I}_{\Gamma_H}=\<t^{|w|-|z|}y^z - y^w : z, w \in \mathsf Z(\gamma), \gamma \in \Gamma, |z| \le |w|\> \subset \kk[t, y_1, \ldots, y_k].\] 
Algorithm~\ref{a:affinedeltagb} exploits this idea to compute the generators $I_j$ in Theorem \ref{t:deltahilbert}.  As there is no need to compute a Graver basis (see Remark~\ref{r:graverminpres}), Algorithm~\ref{a:affinedeltagb} is better equipped than Algorithm~\ref{a:affinedeltahilbert} to handle input with large generators (see Example~\ref{e:affinedeltahilbert}).

\begin{alg}\label{a:affinedeltagb}
Computes the delta set of an affine semigroup $\Gamma = \<\alpha_1, \ldots, \alpha_k\> \subset \NN^d$.  
\begin{algorithmic}
\Function{DeltaSetOfAffineSemigroup}{$\Gamma$}
\State $\rho \gets$ minimal presentation for $\Gamma_H = \<(1,0),(1,\alpha_1), \ldots, (1,\alpha_k)\>$
\State $G \gets$ reduced lex Gr\"obner basis for $\<t^iy^z - t^jy^w : ((i,z), (j,w)) \in \rho\> \subset \kk[t, y_1, \ldots, y_k]$
\State \Return $\{j : t^jy^z - y^w \in G\}$
\EndFunction
\end{algorithmic}
\end{alg}

\begin{thm}\label{t:affinedeltagb}
Fix $\Gamma = \<\alpha_1, \ldots, \alpha_k\> \subset \NN^d$ an affine semigroup. 
%, and let
%\[
%J = \<t^{|w|-|z|}y^z - y^w : z, w \in \mathsf Z(\gamma), \gamma \in \Gamma, |z| \le |w|\> \subset \kk[t, y_1, \ldots, y_k]
%% \Gamma_H = \<(1,0),(1,\alpha_1), \ldots, (1,\alpha_k)\>
%\]
%denote the homogenization of $I_\Gamma$.  
Let $G$ denote a reduced Gr\"obner basis for $\mathrm{I}_{\Gamma_H}$ with respect to any lexicographic term order satisfying $t > y_i$ for each $i \le k$, and for $j \ge 0$ let
\[
J_j = \<y^z - y^w : t^iy^z - y^w \in G, i \le j\>.
\]
We have $I_j = J_j$ for each $j \ge 0$. 
% In particular, Algorithm~\ref{a:affinedeltagb} returns the delta set of $\Gamma$.  
\end{thm}

\begin{proof}
Each generator of $J_j$ certainly lies in $I_j$, so $J_j \subset I_j$.  Conversely, fix $y^z - y^w \in I_j$ with $|z| \le |w|$, and let $i = |w| - |z| \le j$.  By definition, $t^iy^z - y^w \in \mathrm{I}_{\Gamma_H}$, so performing polynomial long division by $G$ yields a remainder of 0.  Since $G$ was computed using a lex term order, the only elements of $G$ used in the division algorithm have leading term dividing $t^iy^z$, and thus correspond to generators of $J_j$.  As such, $I_j \subset J_j$.  The final claim follows from $G$ being reduced, since the initial term $t^jy^z$ of each $t^jy^z - y^w \in G$ cannot be divisible by the leading term of any other elements of $G$, including those corresponding to generators of $J_{j-1}$.  
\end{proof}

\begin{cor}\label{c:affinedeltagb}
Fix $\Gamma = \<\alpha_1, \ldots, \alpha_k\> \subset \NN^d$ an affine semigroup. If $G$ is a reduced Gr\"obner basis for $\mathrm{I}_{\Gamma_H}$ with respect to any lexicographic term order satisfying $t > y_i$ for $i \le k$, then 
\[ 
\Delta(\Gamma)=\{ j \colon t^jy^z-y^w\in G\}.
\]
\end{cor}

\begin{remark}\label{r:affinedeltagb}
In addition to ensuring the correctness of Algorithm~\ref{a:affinedeltagb}, Corollary~\ref{c:affinedeltagb} is also interesting from a theoretical perspective, as it expresses the delta set in terms of Gr\"obner bases of homogeneous toric ideals.  The authors are optimistic that this novel connection will spur interest in the delta set from members of the computational algebra community.  
\end{remark}

%%%%%%%%%%%%%%%%%%%%%%%%%%%%%%%%%%%%%%%%%%%%%%%%%%%%%%%%%%%%%%%%%%%%%%%%%
\subsection{Implementation and benchmarks}%%%%%%%%%%%%%%%%%%%%%%%%%%%%%%%
%raggedbottom%%%%%%%%%%%%%%%%%%%%%%%%%%%%%%%%%%%%%%%%%%%%%%%%%%%%%%%%%%%%

We implemented both algorithms above in \texttt{GAP} to compare their performance to existing methods.  In our implementation of Algorithm~\ref{a:affinedeltahilbert}, we used \texttt{Normaliz} \cite{normaliz} through \texttt{NormalizInterface} \cite{ni} for the calculation of the Hilbert basis $H$, and \texttt{Singular} \cite{singular} via the \texttt{GAP} package \texttt{singular} \cite{singular-gap} for the computation of successive Gr\"obner bases.  Although we did some experiments with \texttt{4ti2} \cite{4ti2} via \cite{4ti2gap} for the calculation of $H$, the resulting implementation was slower than with \texttt{normaliz}; additionally, one cannot use \texttt{4ti2} for Gr\"obner basis computations here since the ideals are not toric. 

\begin{example}\label{e:affinedeltahilbert}
For numerical semigroups, the algorithm used in the \texttt{numeri\-cal\-sgps} \cite{numericalsgps} package is the one presented in \cite{dynamicalg}, which was the fastest known prior to those given above.  Though it only works with numerical semigroups, we can still compare how it behaves against Algorithm \ref{a:affinedeltahilbert}.  A log-time algorithm has also been given for numerical semigroups with 3~minimal generators~\cite{deltadim3}, but the techniques used therein do not easily generalize.  

In examples with ``small'' generators, Algorithm~\ref{a:affinedeltahilbert} outperforms \texttt{DeltaSet\-Of\-Nume\-rical\-Semigroup} (the current implementation in \texttt{numericalsgps}).  For instance, the calculation of $\Delta(\<101,123,147,199\>)$ takes 550~ms with the current implementation, compared to 50~ms with Algorithm \ref{a:affinedeltahilbert}.  Additionally, the current implementation takes 415263~ms to compute $\Delta(\<101,123,147,199,210\>)$, while Algorithm \ref{a:affinedeltahilbert} accomplishes this task in under 4 seconds. 

For semigroups with large generators, the situation can differ.  Choosing, for instance, $\Gamma = \<1001, 1211, 1421, 1631, 2841\>$, \texttt{DeltaSetOfNumericalSemigroup} computes $\Delta(\Gamma)$ in less than~two~minutes, whereas the processes running \texttt{Normaliz} completed Algorithm~\ref{a:affinedeltahilbert} in around three minutes, even when run with \texttt{SCIP}~\cite{scip} (\texttt{4ti2} was killed after half an hour).  
% The computation of $H$ simply becomes too heavy in such cases. 
% Winfried Bruns has also reported that $H$ computes in under three minutes with \texttt{normaliz} and \texttt{SCIP}~\cite{scip}.
\end{example}

One of the advantages of Algorithm~\ref{a:affinedeltagb} is that one can use either \texttt{Singular} (\texttt{eliminate()}) or \texttt{4ti2} (\texttt{groebner()}) to compute a Gr\"obner basis for $\mathrm{I}_{\Gamma_H}$, with comparable performance.  

\begin{example}\label{e:affinedeltagb}
All delta sets appearing in Example \ref{e:affinedeltahilbert} were computed with Algorithm~\ref{a:affinedeltagb} in less than 20~ms.  Even for significantly larger examples such as 
% \[
% \Gamma = \<440, 477, 516, 3066, 5440, 5678, 5690, 5755\>
% \]
% finish computing $\Delta(\Gamma) = \{1, 2, 3, 4, 5, 6, 7, 8, 9, 10, 11, 12, 16, 17, 21, 22\}$ in under a minute.  
% \[
% \Gamma = \<666, 1001, 1101, 2098, 2950, 3512, 3734, 3794, 4239, 5448, 5519, 5541, 5800, 5913, 6275, 7532, 7988\>
% \]
% finish computing $\Delta(\Gamma) = \{1, 2, 3, 4, 5, 6, 7, 8, 9, 10, 11, 12, 14\}$ in under a minute.  
\[
\Gamma = \<550, 1060, 1600, 1781, 4126, 4139, 4407, 5167, 6073, 6079, 6169, 7097, 7602, 8782, 8872\>,
\]
$\Delta(\Gamma) = \{1, 2, 3, 4, 5, 6, 7, 8, 9, 10, 11, 12, 13, 14, 15, 16, 17, 19\}$ computes in less than a minute.  
\end{example}

%%%%%%%%%%%%%%%%%%%%%%%%%%%%%%%%%%%%%%%%%%%%%%%%%%%%%%%%%%%%%%%%%%%%%%%%%
%\subsection{Implementation and benchmarks}%%%%%%%%%%%%%%%%%%%%%%%%%%%%%%%
%raggedbottom%%%%%%%%%%%%%%%%%%%%%%%%%%%%%%%%%%%%%%%%%%%%%%%%%%%%%%%%%%%%

%: Include general discussion of implementation details.  
%
%
%
%Primary computations using existing software: Hilbert bases, Gr\"obner bases, and minimal presentations (which use Gr\"obner bases).  
%
%Packages to discuss: \texttt{numericalsgps}, \texttt{Singular}, \texttt{4ti2}, and \texttt{Normaliz}.  
%
%: not using \texttt{4ti2} for GB computation in Algorithm~\ref{a:affinedeltahilbert} since ideals not saturated!  

% Primary computation intensity is the Hilbert basis computation.  Gr\"obner basis computation does not increase the size of the generating set (hard to make precise).  Approximate statement: The reduced GB has same size under all graded orderings, each ``adjustment'' from Hilbert basis input.  

%\begin{example}\label{e:delta4ti2vsnormaliz}
%: \texttt{4ti2} vs.\ \texttt{Normaliz} runtime comparison.  \texttt{Normaliz} is faster!  
%\end{example}
%
%\begin{example}\label{e:deltanumericalcomparison}
%: Runtime comparison to existing numerical semigroup algorithms.  The new one is faster!  
%\end{example}

%%%%%%%%%%%%%%%%%%%%%%%%%%%%%%%%%%%%%%%%%%%%%%%%%%%%%%%%%%%%%%%%%%%%%%%%%
\section{The catenary degree: a dynamic algorithm}%%%%%%%%%%%%%%%%%%%%%%%
\label{sec:catenarydegree}%%%%%%%%%%%%%%%%%%%%%%%%%%%%%%%%%%%%%%%%%%%%%%%
%raggedbottom%%%%%%%%%%%%%%%%%%%%%%%%%%%%%%%%%%%%%%%%%%%%%%%%%%%%%%%%%%%%

The catenary degree is a factorization invariant that measures how spread out the factorizations of an element are.  Recent results prove the catenary degree is eventually periodic for numerical semigroups \cite{catenaryperiodic}, and a natural generalization is given for affine semigroups \cite{factorhilbert}, but neither result gives a concrete bound for the start of periodic behavior (in contrast to similar results for several other invariants; see~\cite[Remark~5.11]{numericalsurvey}).  In this section, we give Algorithm~\ref{a:dynamiccatenary} for computing the catenary degree of affine semigroup elements using dynamic programming, answering \cite[Problem~5.12]{numericalsurvey}.  The primary motivation for such an algorithm is to aid the investigation of concrete bounds for the aforementioned eventual behavior results, and our approach mirrors several existing algorithms for other factorization invariants \cite{dynamicalg}.  

\begin{defn}\label{d:catenarydegree}
Fix $\Gamma = \<\alpha_1, \ldots, \alpha_k\> \subset \NN^d$ and $\gamma \in \Gamma$.  The \emph{greatest common divisor} of $z, w \in \mathsf Z(\gamma)$ is given by
$$\gcd(z,w) = (\min(z_1,w_1), \ldots, \min(z_k,w_k)) \in \NN^k,$$
and the \emph{distance} between $z$ and $w$ (or the \emph{weight} of $(z,w)$) is defined as 
$$\dist(z,w) = \max(|z - \gcd(z,w)|,|w - \gcd(z,w)|).$$
Given $z, w \in \mathsf Z(\gamma)$ and $N \ge 0$, an \emph{$N$-chain} from $z$ to $w$ %(or a chain of weight $N$) 
is a sequence $z_0, \ldots, z_r \in \mathsf Z(\gamma)$ of factorizations of $\gamma$ such that (i)~$z_0 = z$, (ii)~$z_r = w$, and (iii)~$\dist(z_{i-1},z_i) \le N$ for all $i \le r$.  The~\emph{weight} of $z_0, \ldots, z_r$ is the smallest $N$ such that $z_0, \ldots, z_r$ is an $N$-chain (or, equivalently, the largest distance between successive factorizations).  The \emph{catenary degree} of $\gamma$, denoted $\mathsf c(\gamma)$, is the smallest $N$ such that there exists an $N$-chain between any two factorizations of $\gamma$.  
% The \emph{set of catenary degrees of $S$} is the set $\mathsf C(S) = \{\mathsf c(m) : m \in S\}$.  
\end{defn}

\begin{example}\label{e:catenarydegree}
Consider the numerical semigroup $\Gamma = \<11,36,39\>$.  The left-hand picture in Figure~\ref{fig:catenarydegree} depicts $\mathsf Z(450)$ and all pairwise distances.  Any two factorizations of 450 are connected by a 16-chain; one such chain between $z = (6,2,8)$ and $w = (24,3,2)$ is depicted with bold red edges.  Every 16-chain between $z$ and $w$ contains the bottom edge, so $\mathsf c(450) = 16$.  

The catenary degree can also be computed by examining edges with weight at most 16, as depicted in the right-hand picture of Figure~\ref{fig:catenarydegree}.  Removing edges labeled 8 and 12 yields a (minimal total weight) spanning tree, at which point Lemma~\ref{l:spanningtreecatenary} implies $\mathsf c(450) = 16$.  
\end{example}

\begin{figure}
\includegraphics[width=2.5in]{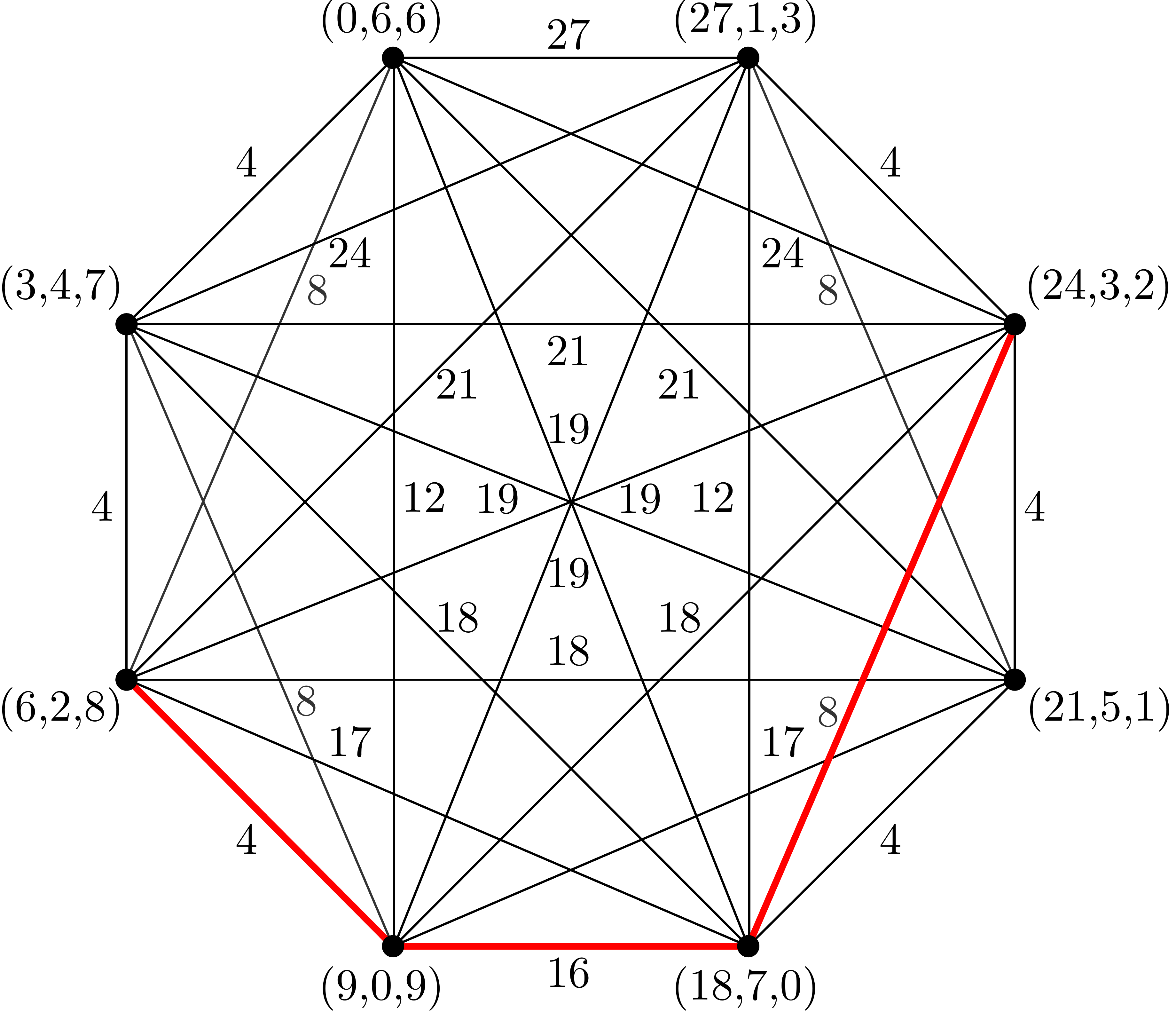}
\hspace{0.5in}
\includegraphics[width=2.5in]{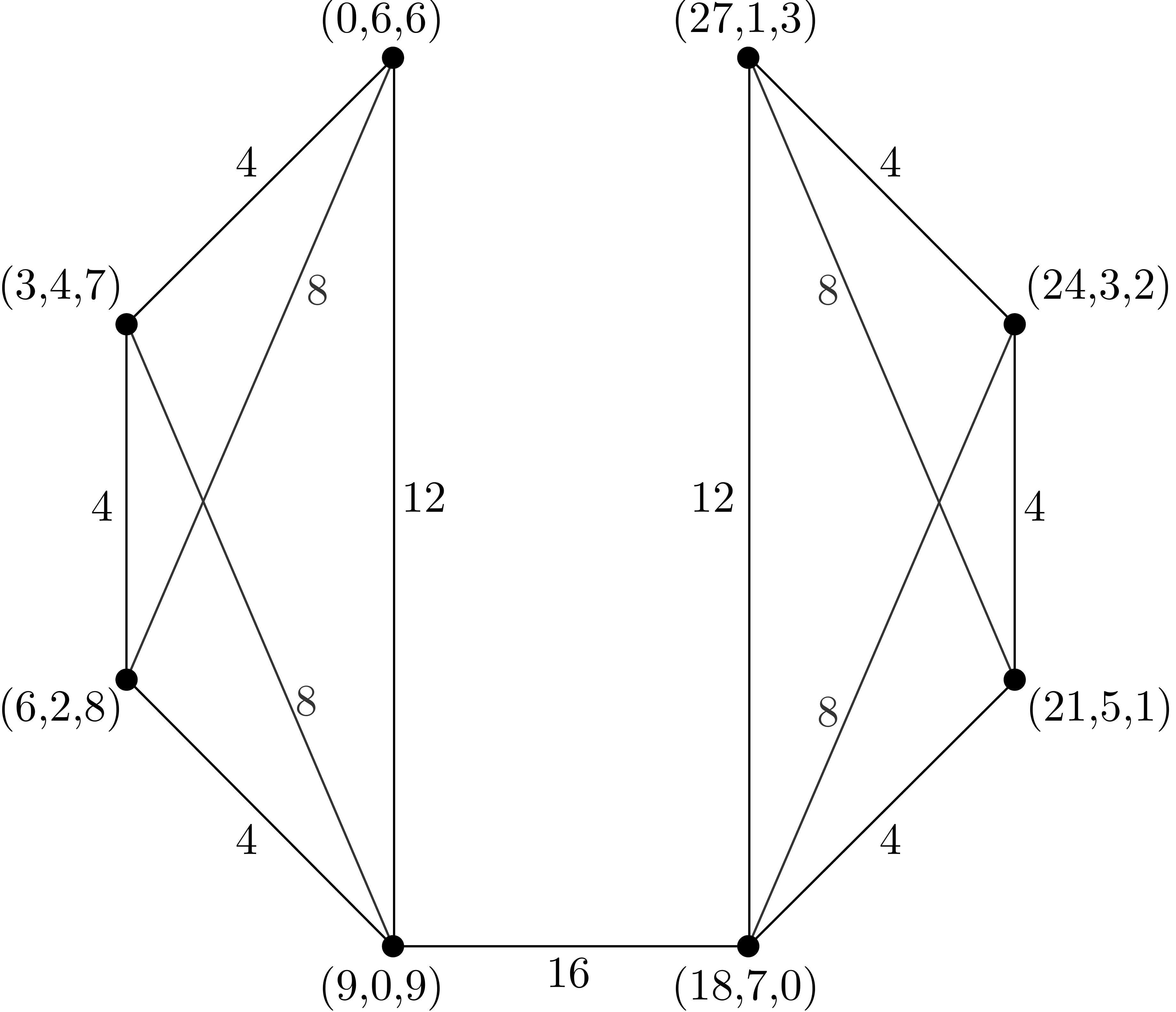}
\medskip
\caption{The catenary graph $K_{450}$ in $\Gamma = \<11,36,39\>$ from Example~\ref{e:catenarydegree}.  Both graphics were created using the computer algebra system \texttt{SAGE} \cite{sage}.}  
\label{fig:catenarydegree}
\end{figure}

The graph theoretic description of the catenary degree evident in Example~\ref{e:catenarydegree} is one of the core structures used in Algorithm~\ref{a:dynamiccatenary}.  We now make this definition explicit.  

\begin{defn}\label{d:catenarygraph}
% Let $e_i$ denote the $i$-th standard basis vector in $\NN^k$.  
Given $\gamma \in \Gamma$, the \emph{catenary graph} of $\gamma$ is the complete graph $K_\gamma$ with $V(K_\gamma) = \mathsf Z(\gamma)$ whose edges are labeled by the distance function $\dist(-,-)$.  
\end{defn}

To compute the catenary degree of an element $\gamma \in \Gamma$, one typically first computes $\mathsf Z(\gamma)$, and compiles a list of edges in $K_\gamma$ sorted by distance.  From this list of edges, a minimal connected graph is constructed one edge at a time, starting with the lowest weight (one can use Kruskal's Algorithm \cite{kruskal} for this).  The final edge weight equals the catenary degree.  One of the primary bottlenecks in this procedure is listing and sorting the edges of $K_\gamma$, since for large $\gamma$ most edges have weight higher than the catenary degree and thus are never used.  
% Asymptotically, $|\mathsf Z(\gamma)| \approx |\gamma|^{k-1}$ \cite{factorasymp}, so $K_\gamma$ has $|\gamma|^{2k-2}$ edges (which must be sorted).  

The main goal of Algorithm~\ref{a:dynamiccatenary} is to reduce the number of edges by remembering which edges were redundant for smaller elements.  This is achieved by storing, for each $\gamma \in \Gamma$, a spanning tree of the graph $K_\gamma$ with minimum weight (Definition~\ref{d:minimalweight}), from which the catenary degree can be easily recovered (Lemma~\ref{l:spanningtreecatenary}).  Indeed, the number of edges of $K_\gamma$ is quadratic in $|V(K_\gamma)|$, whereas the number of edges of a spanning tree is linear in $|V(K_\gamma)|$.  

\begin{defn}\label{d:minimalweight}
Let $G$ be an undirected graph with integer edge labels.  We say a chain $q = v_0, v_1, \ldots, v_r \in V(G)$ has \emph{minimum weight} if $q$ has minimum weight among all chains in $G$ connecting $v_0$ and $v_r$.  A spanning tree $T$ of $G$ has \emph{minimum weight} if the unique path between any two vertices $v_0$ and $v_r$ in $T$ is a minimum weight path in $G$.  
\end{defn}

\begin{lemma}\label{l:spanningtreecatenary}
If $T$ is a minimum weight spanning tree of the catenary graph of $\gamma \in \Gamma$, then $\mathsf c(\gamma)$ equals the highest edge weight in $T$.  
\end{lemma}

\begin{proof}
Immediate upon unraveling definitions.  
\end{proof}

The term ``minimum weight spanning tree'' usually refers to a spanning tree that minimizes edge weight sums.  Proposition~\ref{p:minspanningtree} verifies that our definition is equivalent to the usual one.  

\begin{prop}\label{p:minspanningtree}
Fix a graph $G$, a spanning tree $T \subset G$ with $E(T) = \{e_1, \ldots, e_r\}$, and a weight function $w$ on the edges of $G$. If~$w(e_1) + \cdots + w(e_r)$ is minimal among all spanning trees of $G$, then $T$ is a minimum weight spanning tree of $G$.  
\end{prop}

\begin{proof}
Let $p$ denote a minimum weight chain $v_0, v_1, \ldots, v_r \in V(G)$, and let $q$ denote the unique path in $T$ from $v_0$ to $v_r$.  For $i \le r$, let $q_i$ denote the unique path in $T$ connecting $v_{i-1}$ and $v_i$.  Concatenating the chains $q_1, q_2, \ldots, q_r$ produces a chain $q'$ in $T$ connecting $v_0$ and $v_r$.  Since $q$ is the unique path in $T$ connecting $v_0$ and $v_r$, every edge in $q$ must also be in $q'$.  In particular, the edge $e$ of highest weight in $q$ must lie in one of the chains $q_m$.  As such, if $w(p) < w(q)$, then $w(v_{m-1}, v_m) \leq w(p) < w(q) = w(e)$, so deleting $e$ from $T$ and adding $(v_{m-1}, v_m)$ produces a spanning tree of lesser weight.  This completes the proof.  
\end{proof}

Now that we have identified minimum weight spanning trees as sufficient for computing the catenary degree, it remains to see how a minimum weight spanning tree for $K_\gamma$ can be constructed from minimum weight spanning trees of its divisors.  We accomplish this task using cover morphisms (Definition~\ref{d:covermorphism}), which played a crucial role in the dynamic algorithms presented in \cite{dynamicalg}.  The key observation is that cover morphisms also preserve edge weights when applied to adjacent factorizations in $K_{\gamma - \alpha_i}$ (Remark~\ref{r:covermorphism}).  

\begin{defn}\label{d:covermorphism}
A \emph{cover morphism} is a map $\psi_i:K_{\gamma - \alpha_i} \to K_\gamma$ that acts on the vertex set $V(K_{\gamma - \alpha_i}) = \mathsf Z(\gamma - \alpha_i)$ by incrementing the $i$-th coordinate of each factorization.  
\end{defn}

\begin{remark}\label{r:covermorphism}
When applied to a pair $(z,w)$ of factorizations in $K_{\gamma - \alpha_i}$, the cover morphism $\psi_i$ applies the same translation to both $z$ and $w$, so the resulting factorizations $(z',w')$ satisfy 
\[
z' - \gcd(z',w') = z - \gcd(z,w) \quad \text{ and } \quad w' - \gcd(z',w') = w - \gcd(z,w).  
\]
In particular, $\dist(z',w') = \dist(z,w)$, so $\psi_i$ preserves edge weights (stated another way, $\psi_i$ defines an isometric embedding of $K_{\gamma - \alpha_i}$ into $K_\gamma$).  Additionally, an edge $(z',w') \in E(K_\gamma)$ lies in the image of $\psi_i$ precisely when $i \in \supp(z') \cap \supp(w')$.  Here, $\supp(z)$ denotes the \emph{support} of $z$, that is, the set of indices of the nonzero coordinates of $z$.
\end{remark}

As a final step, Theorem~\ref{t:spanningtree} implies that in addition to edges arising from the minimum weight spanning trees of $K_{\gamma - \alpha_1}, \ldots, K_{\gamma - \alpha_k}$, some supplementary edges from the Graver basis of $\Gamma$ are needed to obtain a minimum weight spanning tree for $K_\gamma$.  The edges from the Graver basis of $\Gamma$ can be thought of as a base case for the dynamic process, in the sense that every edge used in a minimum weight spanning tree is a translation of a Graver basis element.  Although the Graver basis of $\Gamma$ is computationally expensive (albeit necessary; see~Example~\ref{e:catenarygraver}), it only needs to be calculated once for the monoid.  

\begin{example}\label{e:catenarygraver}
Although a minimal presentation for $\Gamma$ produces enough edges via translation to connect $K_\gamma$ for every $\gamma \in \Gamma$, these edges are not sufficient for computing catenary degrees.  For example, in the numerical semigroup $\Gamma = \<11, 36, 39\>$ from Example~\ref{e:catenarydegree}, the only edge $e = ((9,7,0), (0,0,9))$ in $K_{351}$ with weight $\mathsf c(351) = 16$ does not appear in any minimal presentation.  In fact, the weight 16 edge in Figure~\ref{fig:catenarydegree} is a translation of $e$.  It is this phenomenon that forces Algorithm~\ref{a:dynamiccatenary} to rely on a Graver basis for $\Gamma$.  
\end{example}

\begin{thm}\label{t:spanningtree}
For each $i \le k$, let $T_i \subset K_\gamma$ denote the image in $K_\gamma$ of a minimum weight spanning tree of $K_{\gamma - \alpha_i}$ under the cover morphism $\psi_i:K_{\gamma - \alpha_i} \to K_\gamma$.  
% such that $\gamma - \alpha _i \in \Gamma$, let $Z_i = \{w + e_i : f \in \mathsf Z(\gamma - \alpha_i)\}\subset\mathsf{Z}(a)$.  Then, let $G_i=(Z_i, E_i)$ be a complete subgraph of $G$, with minimum spanning tree $T_i = \left(Z_i, E_i'\right)$.  
Additionally, let
\[
E' = \mathcal H(\{(z, w): z_1\alpha_1 + \cdots + z_k\alpha_k = w_1\alpha_1 + \cdots + w_k\alpha_k\}) \cap E(K_\gamma)
\]
be the elements of the Graver basis of $\Gamma$ that lie in $\mathsf Z(\gamma) \times \mathsf Z(\gamma)$.  
The graph 
\[
G' = (\mathsf Z(\gamma), E(T_1) \cup \cdots \cup E(T_k) \cup E')
\]
is connected, and any minimum weight spanning tree of $G'$ is minimum weight in $K_\gamma$.  
\end{thm}

\begin{proof}
Fix $z, w \in \mathsf Z(\gamma)$.  If $j \in \supp(z) \cap \supp(w)$, then a path in $T_j$ connects $z$ and $w$.  
% Similarly, if $z$ and $w$ have disjoint support but are connected by a chain $z = z_0, \ldots, z_r = w$ with $\supp(z_{i-1}) \cap \supp(z_i) \ne \emptyset$ for each $i \le r$, then $z_{i-1}$ and $z_i$ are connected by a chain in $G'$ for each $i \le r$, and appending their respective chains yields a chain connecting $z$ and $w$.  
Alternatively, if $z$ and $w$ have disjoint support and $(z,w) \notin E'$, then $(z,w)$ can be expressed as a sum
$(z,w) = (z_1,z_1') + \cdots + (z_r,z_r')$
of Graver basis elements, so the chain $q$ given by
\[
z_1 + z_2 + \cdots + z_r, \quad z_1' + z_2 + \cdots + z_r, \quad \ldots, \quad z_1' + z_2' + \cdots + z_r'
\]
connects $z$ and $w$ in $K_\gamma$.  Since each edge in $q$ connects factorizations with common support, $z$ and $w$ are connected in $G'$.  

Next, let $T'$ be any minimum weight spanning tree of $G'$. %, and let $T$ be any spanning tree of $K_\gamma$.  
It suffices to prove that the~unique~path $q'$ in $T'$ connecting $(z, w) \in E(K_\gamma) \setminus E(T')$ has weight at most $\dist(z,w)$.  If $j \in \supp(z) \cap \supp(w)$, 
% (or if $z$ and $w$ are connected by a chain of edges with this property) 
then the unique path from $z$ to $w$ in $T_j$ has weight at most $\dist(z,w)$, and thus so does $q'$.  On the other hand, if $z$ and $w$ have disjoint support, then the $i$-th edge in the chain $q$ constructed in the above paragraph has weight
\[
\dist(z_i,z_i') = \max(|z_i|,|z_i'|) \le \max(|z|,|w|) = \dist(z,w),
\]
meaning $q'$ has weight at most $\max \{\dist(z_i, z_i') : i \le r\} \le \dist(z,w)$.  This completes the proof.  
\end{proof}

We are now ready to state Algorithm~\ref{a:dynamiccatenary}.  

\begin{alg}\label{a:dynamiccatenary}
Finds $\mathsf c(\gamma)$ for any affine semigroup element $\gamma \in \Gamma = \<\alpha_1, \ldots, \alpha_k\> \subset \NN^d$.  
\begin{algorithmic}
\Function{CatenaryDegreeOfAffineSemigroupElement}{$\gamma,\Gamma$}
\State $T \gets $ \Call{MinimumWeightSpanningTree}{$\gamma,\Gamma$}
\State \Return $\max\{\dist(z,w) : (z,w) \in E(T)\}$.
\EndFunction
\State
\State $H \gets \mathcal H(\{(z, w) : z_1\alpha_1 + \cdots + z_k\alpha_k = w_1\alpha_1 + \cdots + w_k\alpha_k\})$
\Function{MinimumWeightSpanningTree}{$\gamma, \Gamma$}
\If{$T_\gamma$ not already computed}
	\If{$|\mathsf Z(\gamma)| \le 1$}
		\State $T_\gamma \gets (\mathsf Z(\gamma), \emptyset)$
	\Else
		\State $E' \gets H \cap E(K_\gamma)$
		\ForAll{$j = 1, \ldots, k$}
			\State $T_i \gets \psi_i(\Call{MinimumWeightSpanningTree}{\gamma - \alpha_i,\Gamma})$
		\EndFor
		\State $T_\gamma \gets \Call{Kruskal}{\mathsf Z(\gamma), E(T_1) \cup \cdots \cup E(T_k) \cup E'}$
	\EndIf
\EndIf
\State \Return $T_\gamma$
\EndFunction
\end{algorithmic}
\end{alg}

\begin{thm}\label{t:dynamiccatenary}
For any affine semigroup element $\gamma \in \Gamma$, Algorithm~\ref{a:dynamiccatenary} returns $\mathsf c(\gamma)$.  
\end{thm}

\begin{proof}
Proposition~\ref{p:minspanningtree} and Theorem~\ref{t:spanningtree} ensure that \texttt{MinimumWeightSpanningTree} returns a minimum weight spanning tree $T \subset K_\gamma$, and by Lemma~\ref{l:spanningtreecatenary}, the maximum weight among the edges of $T$ equals $\mathsf c(\gamma)$.  
\end{proof}

%%%%%%%%%%%%%%%%%%%%%%%%%%%%%%%%%%%%%%%%%%%%%%%%%%%%%%%%%%%%%%%%%%%%%%%%%
\subsection{Implementation and benchmarks}%%%%%%%%%%%%%%%%%%%%%%%%%%%%%%%
%raggedbottom%%%%%%%%%%%%%%%%%%%%%%%%%%%%%%%%%%%%%%%%%%%%%%%%%%%%%%%%%%%%

Algorithm~\ref{a:dynamiccatenary} has now been implemented as \texttt{CatenaryDegreeElementListWRTNumericalSemigroup} in the \texttt{numericalsgps} \cite{numericalsgps} \texttt{GAP} \cite{gap} package.  The function definition mirrors that of the dynamic length sets and $\omega$-primality functions \cite{dynamicalg}, which motivated the results in this section.  A version for affine semigroups is forthcoming, pending discussion of the function definition.  
% (Actually several depending on the packages installed by the user).  
The code is accessible in \cite{numericalsgps} or from the official \texttt{GAP} web page \url{http://www.gap-system.org}. The development version of \texttt{numericalsgps} can be found in \url{https://bitbucket.org/gap-system/numericalsgps}. 

We give a few implementation notes before concluding the section with an example.  

\begin{enumerate}
\item 
If $\Gamma = \<\alpha_1, \ldots, \alpha_k\> \subset \NN$ is a numerical semigroup, then it is natural to begin with $\gamma = 0$ and compute $\mathsf c(\gamma)$ in order.  This allows the spanning trees to be stored in a ring buffer of size $\alpha_k$ (the largest generator), since $\gamma - \alpha_k$ is the smallest element that needs to be recalled to compute $\mathsf c(\gamma)$.  

\item 
By storing the list of edges in $E(T_\gamma)$ ordered by weight and computing the union $E(T_1) \cup \cdots \cup E(T_k) \cup E'$ via a merge of sorted lists, the first (and most computationally-intensive) step in Kruskal's Algorithm runs in linear time.  

\item 
At no point must $\mathsf Z(\gamma)$ be explicitly computed.  The base case $\mathsf Z(\gamma) \le 1$ consists of precisely those elements for which $E(T_1)$, \ldots, $E(T_k),$ and $E'$ are all empty, and in all other cases, $\mathsf Z(\gamma)$ is simply the set of edge endpoints by Theorem~\ref{t:spanningtree}.  This is essentially~\cite[Algorithm~3.3]{dynamicalg} for dynamically computing sets of factorizations.  

\end{enumerate}

\begin{example}\label{e:catenarynum}
Given here are runtimes for three separate methods of computing $\mathsf c(\gamma)$ in $\Gamma = \<11,23,27,31,43\> \subset \NN$ for $\gamma \le 500$.  

\begin{verbatim}
gap> s:=NumericalSemigroup(11,23,27,31,43);;
gap> c:=CatenaryDegreeOfNumericalSemigroup(s);;
gap> l:=Intersection([1..500], s);;
gap> lc:=List(l, x->CatenaryDegreeOfElementInNumericalSemigroup(x,s));; time;
270390
gap> ln:=List(l, x->CatenaryDegreeOfElement(x,s,c));; time;
29380
gap> ld:=CatenaryDegreeElementListWRTNumericalSemigroup(l,s);; time;
6630
\end{verbatim}

The function \texttt{CatenaryDegreeOfElementInNumericalSemigroup} (original implementation in the package) computes $\mathsf c(\gamma)$ by computing the complete graph $K_\gamma$ and then removing edges by weight, whereas the function \texttt{CatenaryDegreeOfElement} constructs a spanning tree for $K_\gamma$ using Kruskal's Algorithm (as discussed immediately following Definition~\ref{d:catenarygraph}).  The latter will replace the implementation of the former in the next release of the \texttt{numericalsgps} package.  Although \texttt{CatenaryDegreeOfElement} is an improvement over the original implementation, the dynamic algorithm is still considerably faster for large $\gamma$.  
\end{example}

%%%%%%%%%%%%%%%%%%%%%%%%%%%%%%%%%%%%%%%%%%%%%%%%%%%%%%%%%%%%%%%%%%%%%%%%%
\section{The tame degree}%%%%%%%%%%%%%%%%%%%%%%%%%%%%%%%%%%%%%%%%%%%%%%%%
\label{sec:tamedegree}%%%%%%%%%%%%%%%%%%%%%%%%%%%%%%%%%%%%%%%%%%%%%%%%%%%
%raggedbottom%%%%%%%%%%%%%%%%%%%%%%%%%%%%%%%%%%%%%%%%%%%%%%%%%%%%%%%%%%%%

A variant of the catenary degree, the tame degree also uses distance (as in Definition~\ref{d:catenarydegree}) to measure how wild the factorizations of monoid elements are~ \cite{g-hk}.  In this section, we present Algorithm~\ref{a:affinetame}, an improved method for computing the tame degree of a full affine semigroup (Definition~\ref{d:full}).  We begin by presenting Theorem~\ref{t:tameprincipal}, the main new theoretical result used in Algorithm~\ref{a:affinetame}, followed by a discussion of the algorithm.  One of the primary motivations for having a dedicated algorithm for this family of affine semigroups is to examine the tame degree of block monoids; in Examples~\ref{e:tameb2x2x2}-\ref{e:tameb2x2x2x2}, we apply our algorithm in this setting.  

\begin{defn}\label{d:tamedegree}
Fix $\Gamma = \<\alpha_1, \ldots, \alpha_k\> \subset \NN^d$.  Given $\gamma \in \Gamma$ and $i \le k$ such that $\gamma - \alpha_i \in \Gamma$, the \emph{tame degree} of $\gamma$ with respect to $\alpha_i$, denoted $\mathsf t_i(\gamma)$, is defined as the least $N \in \NN \cup \{\infty\}$ such that for every $z \in \mathsf Z(\gamma)$, there exists $w \in \mathsf Z(\gamma)$ with $w_i > 0$ and $\dist(z,w)\le N$.  
% Notice that the condition $\mathsf Z(m) \cap (x + \mathcal F) \neq \emptyset$ is equivalent to $\gamma \in \varphi(x) + \Gamma$.  
The \emph{tame degree} of $\Gamma$ with respect to $\alpha_i$ is then defined as 
\[
\mathsf t_i(\Gamma) = \sup\{\mathsf t_i(\gamma) : \gamma - \alpha_i \in \Gamma\}.
\]
% and $\Gamma$ is said to be \emph{locally tame} if $\mathsf t_i(\Gamma)$ is finite for all $i \le k$.  
Lastly, the \emph{tame degree} of $\Gamma$ is defined as $\mathsf t(\Gamma) = \max\{\mathsf t_i(\Gamma) \colon i \le k\}$.  
% , and we say that $\Gamma$ is \emph{tame} if its tame degree is finite.
\end{defn}

In what follows, we denote by `$\le$' the usual partial ordering in $\NN^k$, that is, $x \le y$ if $y - x \in \NN^k$, or, equivalently, if $x_i \le y_i$ for all $i \le k$.  Notice that $x\le y$ is also equivalent to $y\in x+\NN^k=\{x+z\colon z\in \NN^k\}$.
% We will simply write $\le$ for $\le_{\NN^k}$ when there is no possible misunderstanding.

\begin{thm}\label{t:tameprincipal}
Let $\Gamma = \<\alpha_1, \ldots, \alpha_k\> \subset \NN^d$, and fix $i \le k$.  We have 
\[
\mathsf t_i(\Gamma) = \sup\{\mathsf t_i(\gamma) : \gamma \in M_i\},
\]
where $M_i = \{\varphi_\Gamma(x) : x\in \mathrm{Minimals}_\le \mathsf Z(\alpha_i + \Gamma)\}$.
\end{thm} 

\begin{proof}
Since $M_i \subseteq \Gamma$, clearly $\sup\{\mathsf t_i(m) : m \in M_i\} \le \mathsf t(\Gamma)$.  Conversely, fix $\gamma \in \alpha_i + \Gamma$ and $z \in \mathsf Z(\gamma)$.  If $z_i > 0$, then choose $w = z$ in the definition of tame degree. Otherwise, $z \in \mathsf Z(\alpha_i + \Gamma)$ since $\gamma \in \alpha_i + \Gamma$, and consequently, there exists $z' \in \mathrm{Minimals}_\le \mathsf Z(\alpha_i + \Gamma)$ such that $z' \le z$.  Set $\gamma' = \varphi_\Gamma(z')$.  By the definition of $\mathsf t_i(-)$, there exists $w' \in \mathsf Z(\gamma') \cap (\alpha_i + \NN^k)$ with $\dist(z',w') \le \mathsf t_i(\gamma')$. Let $w = w' + (z - z')$.  Upon verifying that 
\[
\varphi_\Gamma(w) = \varphi_\Gamma(w') + \varphi_\Gamma(z - z') = \varphi_\Gamma(z') + \varphi_\Gamma(z - z') = \varphi_\Gamma(z) = \gamma,
\]
$\dist(z,w) = \dist(z',w') \le \mathsf t_i(\gamma')$, and $w \in \mathsf Z(\gamma) \cap (\alpha_i + \NN^k)$, it follows that $\mathsf t_i(n) \le \mathsf t_i(m)$.  This implies $\mathsf t_i(\gamma) \le \sup\{\mathsf t_i(\gamma') : \gamma' \in M_i\}$, and consequently $\mathsf t_i(\Gamma) \le \sup\{\mathsf t_i(\gamma') : \gamma' \in M_i\}$.
\end{proof}

\begin{remark}\label{r:tameprincipal}
Theorem~\ref{t:tameprincipal} holds in more generality with a nearly identical proof; the only requirement is that $\Gamma$ be atomic.  We specialize to the affine setting here to simplify notation since Algorithm~\ref{a:affinetame} requires this assumption; see~\cite{g-hk} for the general definitions.  
\end{remark}

\begin{remark}\label{r:tameomega}
Comparing Theorem~\ref{t:tameprincipal} to \cite[Proposition~3.3]{omega}, it is not a coincidence that the elements needed to compute the tame degree of $\Gamma$ are the same as those needed to compute the $\omega$-primality invariant (see~\cite[Definition~3.1]{omega}).  Some evidence of this connection was already observed for numerical semigroups; see, for instance, \cite[Corollary~3]{c-t-ns} and \cite[Remark~5.9]{omega}. 
\end{remark}

Arranging the atoms $\alpha_1, \ldots, \alpha_k$ of $\Gamma$ as the columns of a matrix $A$, the set $\mathsf Z(\gamma)$ of factorizations of $\gamma \in \Gamma$ coincides with the nonnegative integer solutions to the linear system
\[
Ax = \gamma.
\]
% (and some modular equations if $G \ne 0$) 
% If two solutions $x$ and $y$ are comparable in $\NN^k$ (that is, either $y - x \in \NN^k$ or $x - y \in \NN^k$), then $A(y - x) = A(x - y) = 0$ and according to the above paragraph, $x - y$ must be zero.  In light of Dickson's lemma, this means that $\mathsf Z(\gamma)$ has finitely many elements.  
Prior to Theorem~\ref{t:tameprincipal}, the tame degree of $\Gamma$ was typically computed by first computing the Graver basis of $A$ \cite{c-t}.  We propose here an alternative method, instead computing the sets 
\[
\mathcal M_i = \mathrm{Minimals}_\le \mathsf Z(\alpha_i + \Gamma) \quad \text{ and } \quad M_i = \{\varphi_\Gamma(z)\colon z\in \mathcal M_i\}
\]
for each $i\in\{1,\ldots,k\}$.  By Dickson's lemma, both $\mathcal M_i$ and $M_i$ have finitely many elements.

For the calculation of $\mathcal M_i$ one can use \cite[Algorithm 16]{minimales} (this is precisely how $\omega$-primality is computed in \cite{GGMT}).  However, it turns out that \texttt{Normaliz}~\cite{normaliz2, normaliz} performs this task faster.  To obtain $\mathcal M_i$, we first compute the minimal nonnegative integer solutions of 
\[ (A\mid -A)\begin{pmatrix} x\\ y \end{pmatrix} = \alpha_i,
\]
and then project on the first $k$ coordinates (here, $(A\mid -A)$ is the matrix having as columns the columns of $A$ followed by those of $-A$).  From the resulting (finite) set, we simply take those elements that are minimal with respect to $\le$.  

Since the Graver basis of $A$ coincides with the minimal nonnegative integer solutions of 
\[
(A\mid -A)\begin{pmatrix} x\\ y \end{pmatrix} = 0,
\] 
it would seem that our approach has no significant computational advantage over the procedure given in \cite{c-t}.  However, combining with \cite[Corollary~3.5]{omega} (stated here as Proposition~\ref{p:principalsaturated}) yields significant performance improvements for full affine semigroups.  
% Actually \cite[Corollary~3.5]{omega} can be stated in a more general setting, and this is what we do next (Proposition~\ref{p:principalsaturated}). 
% First, we introduce the concept of saturated monoid, which generalizes that of full affine semigroup.

%A particular case of finitely generated saturated monoids are full affine semigroups. 
\begin{defn}\label{d:full}
An affine semigroup $\Gamma \subset \NN^d$ is \emph{full} if $\mathrm G(\Gamma)\cap\NN^d=\Gamma$, where $\mathrm G(\Gamma)$ denotes the subgroup of $\mathbb Z^d$ generated by $\Gamma$.
%it is saturated in $\NN^d$, that is, whenever $n\gamma \in \Gamma$ for $\gamma \in \NN^d$ and some positive integer $n$, we have $\gamma \in \Gamma$.  
\end{defn}

\begin{remark}\label{r:full}
Up to isomorphism, every full affine semigroup is a (reduced finitely generated) Krull monoid \cite[Proposition 2]{k-l}.  Also, the class of reduced Krull monoids coincides with the class of saturated submonoids of free monoids \cite[Theorem 2.4.8]{g-hk}.
\end{remark}

% \begin{defn}\label{d:saturated}
% Fix a submonoid $\Gamma$ of a free monoid $F$.  The \emph{quotient group} of $\Gamma$ in $F$ is the smallest subgroup of $G$ containing $\Gamma$.  
% A submonoid $\Gamma$ a free monoid $F$ is \emph{saturated} if $\mathrm G(\Gamma) \cap F = \Gamma$.  
% %This class of monoids coincides with the class of reduced finitely generated Krull monoids (\cite[Theorem 2.4.8]{g-hk}).
% \end{defn}

% Since we are assuming the $\Gamma$ is cancellative, then we can consider $\mathrm G(\Gamma)$, the quotient group of $\Gamma$. This group can be constructed as $(\Gamma\times \Gamma)/\sim$, where $(m,n)\sim (m',n')$ if $m+n'=n+m'$. The opposite of an element $[(m,n)]$ is $[(n,m)]$. The map $m\mapsto [(m,0)]$ becomes a monomorphism from $\Gamma$ to $\mathrm G(\Gamma)$. Since every $[(m,n)]\in \mathrm G(\Gamma)$ can be expressed as $[(m,0)]+[(0,n)]=[(m,0)]-[(n,0)]$, we can identify (via this monomorphism) the elements of $\mathrm G(\Gamma)$ with $m-n$ with $m,n\in \Gamma$.

\begin{prop}[{\cite[Corollary~3.5]{omega}}]\label{p:principalsaturated}
If $\Gamma = \<\alpha_1, \ldots, \alpha_k\> \subset \NN^d$ is full, then for every $\gamma \in \Gamma$,
\[
\mathsf Z(\gamma + \Gamma)= \{x \in \NN^k : \gamma \le \varphi_\Gamma(x)\}.
\]
\end{prop}

% \begin{proof}
% First, if $x \in \mathsf Z(m+\Gamma)$, then $\varphi(x) \in m + \Gamma \subseteq m + F$.  Conversely, suppose $x \in \mathcal F$ with $m \le_F \varphi(x)$.  Then $\varphi(x)=m+f$ for some $f\in F$. Hence $f=\varphi(x)-m\in \mathrm G(\Gamma)\cap F$. By hypothesis, $\Gamma$ is saturated, and consequently $f\in \Gamma$. This leads to $\varphi(x)\in m+\Gamma$, and therefore $x\in \mathsf Z(m+\Gamma)$.
% \end{proof}

%So let us assume that $M$ is a submonoid of $\mathbb N^e$ for some positive integer $e$ (and consequently an affine semigroup). Let $\mathrm G(M)$ be the subgroup of $\mathbb Z^e$ spanned by $M$. We say that $M$ is \emph{full} (sometimes known as saturated in the literature, see for instance \cite{g-hk}) if $M=\mathbb N^e\cap \mathrm G(M)$. As a particular case of \cite[Corollary 3.5]{omega}, we get the following.

%\begin{prop}\label{principal-full}
%Let $M$ be a full affine semigroup minimally generated by $\{m_1,\ldots, m_t\}$, and assume that these elements are the columns of the matrix $A$. Then 
%\[ \mathcal M_i=\mathrm{Minimals}_\le \left\{ x\in \mathbb N^t \colon A x\ge m_i\right\}.\]
%\end{prop}

\begin{example}
The ``full'' hypothesis in Proposition~\ref{p:principalsaturated} cannot be omitted.  For example, the numerical semigroup $\Gamma = \<3,5,7\> \subset \NN$ is not full since $5 - 3 = 2 \in (\mathrm G(\Gamma) \cap \mathbb N) \setminus \Gamma$ (in~fact, the only full numerical semigroup is $\NN$).  In this example, we see that 
\[
\mathrm{Minimals}_\le \mathsf Z(3 + \Gamma) = \{(0,0,2), (0,1,1), (0,2,0), (1,0,0)\},
\]
even though $\mathrm{Minimals}_\le \{x \in \NN^3 : 3x_1 + 5x_2 + 7x_3 \ge 3\} = \{(1,0,0),(0,1,0),(0,0,1)\}$.  
\end{example}

Recalling $A = (\alpha_1 \mid \cdots \mid \alpha_k)$ from above, Proposition~\ref{p:principalsaturated} states that
\[
\mathsf Z(\gamma + \Gamma) = \{x \in \NN^k : Ax \ge \gamma\}, 
\]
meaning that $\mathcal M_i$ can be computed as a Hilbert basis.  At this point, in order to compute $\mathsf t_i(\Gamma)$ using Theorem~\ref{t:tameprincipal}, it remains to compute $\mathsf t_i(\gamma)$ for each $\gamma \in M_i$.  To this end, 
% For a given element $\gamma \in \alpha_i + \Gamma$, Computing $\mathsf t_i(\gamma)$ we have first to calculate the set of factorizations of $\gamma$, $\mathsf Z(\gamma)$, which was shown above to coincide with the set of nonnegative integer solutions of $Ax = \gamma$.  Then $\mathsf t_i(\gamma)$ is the least nonnegative integer $N$ such that for any $z \in \NN^k$ with $Az = \gamma$, there exists $y \in \NN^k$ with $Ay = \gamma$, $y_i \neq 0$ (the existence of such factorization follows from $\gamma \in \alpha_i + \Gamma$) and $\dist(z,y) \le N$.  Notice that if $z_i \neq 0$, we can simply take $y = z$ and $\dist(z,y)=0$. 
Proposition~\ref{p:fulldist} implies that for fixed $z \in \mathrm{Minimals}_\le \mathsf Z(\alpha_i + \Gamma)$ with $z_i = 0$, we can find the closest factorization $w \in \NN^k$ to~$z$ in the fiber of $\varphi_\Gamma(z)$ with $w_i\neq 0$ by minimizing $|w|$ subject to the constraints $Aw = Az$ and $w \cdot z = 0$.  In particular, this expresses $\mathsf t_i(\gamma)$ as the result of an integer linear programming~problem.  

\begin{prop}\label{p:fulldist}
Let $\Gamma = \<\alpha_1, \ldots, \alpha_k\> \subset \NN^d$, fix $i \le k$,  and fix $z \in \mathrm{Minimals}_\le \mathsf Z(\alpha_i + \Gamma)$.  If~$z_i = 0$ and $w \in \NN^k$ fulfills $w_i \neq 0$ and $Aw = Az$, then $z \cdot w = 0$ and $\dist(z,w) = \max\{|z|,|w|\}$.
\end{prop}

\begin{proof}
Assume to the contrary that $u = \gcd(z, w) \ne 0$.  As $z_i=0$, we deduce that $u_i=0$. This implies that the $i$-th coordinate of $w - u$ is nonzero, and consequently $w - u \in \mathsf Z(\alpha_i + \Gamma)$.  Hence $z - u \in \mathsf Z(\alpha_i+\Gamma)$ and $z - u < z$, contradicting the minimality of $z$.
\end{proof}

The last necessary observation is that $e_i \in \mathrm{Minimals}_\le \mathsf Z(\alpha_i + \Gamma)$, where $e_i$ is the $i$th column vector of the $k \times k$ identity matrix, so it suffices to consider those $z \in \mathrm{Minimals}_\le \mathsf Z(\alpha_i + \Gamma)$ satisfying $z_i = 0$ when computing $\mathsf t_i(\Gamma)$.  
% As $\varphi_\Gamma(e_i) = \alpha_i$, $\mathsf Z(\alpha_i ) = \{e_i\}$.
We now summarize our procedure in Algorithm~\ref{a:affinetame}.  

\begin{alg}\label{a:affinetame}
Computes $\mathsf t_i(\Gamma)$ for a full affine semigroup $\Gamma = \<\alpha_1, \ldots, \alpha_k\> \subset \NN^d$.  
\begin{algorithmic}
\Function{TameDegreeOfFullAffineSemigroup}{$\Gamma$,$i$}
\State $H \gets $ minimal solutions $z$ of $Ax \ge \alpha_i$ with respect to $\le$ satisfying $z_i = 0$
\If{$H$ is empty}
\State \Return 0 (and thus stop)
\EndIf
% (a Hilbert Basis)
\State $P \gets \emptyset$
\ForAll{$z \in H$}
\State minimize $y_1 + \cdots + y_k$ subject to $Ay = Az$ and $y \cdot z = 0$
\State $P \gets P \cup \{y\}$
\EndFor
\State \Return $\max\{|z|, |y| : z \in H, y \in P\}$.
\EndFunction
\end{algorithmic}
\end{alg}

\begin{thm}\label{t:affinetame}
For any full affine semigroup $\Gamma$, Algorithm~\ref{a:affinetame} returns $\mathsf t_i(\Gamma)$.  
\end{thm}

\begin{proof}
Theorem~\ref{t:tameprincipal}, Propositions~\ref{p:principalsaturated} and~\ref{p:fulldist}, and the discussion in between.  
\end{proof}

\begin{remark}
Notice that if all solutions $y$ obtained in Algorithm~\ref{a:affinetame} satisfy $|y| \le |z|$, then $\mathsf t_i(\Gamma)$ will coincide with the $\omega$-primality of $\alpha_i$, $\omega(\alpha_i)$ (see \cite[Proposition 3.3]{omega}).  Indeed, this provides an alternative proof that $\omega(\alpha_i) \le \mathsf t_i(\Gamma)$.  
\end{remark}

\subsection{Implementation and benchmarks}

We have implemented Algorithm~\ref{a:affinetame} in the \texttt{numericalsgps} \cite{numericalsgps} \texttt{GAP} \cite{gap} package.  The function \texttt{TameDegreeOfAffineSemigroup} therein has two main methods: one for affine semigroups in general, and another for full affine semigroups (affine semigroups with the attribute \texttt{HasEquations}).  
% (Actually several depending on the packages installed by the user).  
The code is accessible in \cite{numericalsgps} or from the official \texttt{GAP} web page \url{http://www.gap-system.org}. The development version of \texttt{numericalsgps} can be found in \url{https://bitbucket.org/gap-system/numericalsgps}. 

We conclude this section by applying our implementation of Algorithm~\ref{a:affinetame} in several examples of block monoids (Definition~\ref{d:blockmonoid}).  

\begin{defn}\label{d:blockmonoid}
Fix an abelian group $(G,+)$ and a finite subset $H = \{g_1, \ldots, g_n\} \subset G\setminus\{0\}$.  The \emph{block monoid} associated to $H$ is the affine semigroup
\[
\mathcal B(H) = \{(x_1, \ldots, x_n) \in \NN^n : x_1g_1 + \dots + x_ng_n = 0\} \subset \NN^n.
\]
Denote by $\mathcal B(G)$ the block monoid of the set of nonzero elements of $G$.
% that is, the set of sequences $(x_1, \ldots, x_n) \in \NN^n$ satisfying $x_1g_1 + \dots + x_ng_n = 0$.  
\end{defn}

\begin{remark}\label{r:blockmonoid}
It is easy to show that $\mathcal B(H)$ is a full affine semigroup.  Many factorization properties of a monoid can be derived from the factorization properties of the block monoid associated to its class group (see \cite{g-hk} for a thorough treatment).  
\end{remark}

\begin{example}\label{e:tameb2x2x2}
We begin by computing the tame degree of $\Gamma=\mathcal B(\ZZ_2^3)$.  If $B$ denotes the matrix with columns the nonzero elements of $\ZZ_2^3$, then $\Gamma$ is the submonoid of $\NN^7$ of nonnegative integer solutions of the system 
\[
Bx \equiv 0 \bmod 2.
\]

The tame degree of $\Gamma$ can be computed in \texttt{GAP} as follows.
\begin{verbatim}
gap> m:=[ [ 0, 0, 1 ], [ 0, 1, 0 ], [ 0, 1, 1 ], [ 1, 0, 0 ], 
          [ 1, 0, 1 ], [ 1, 1, 0 ], [ 1, 1, 1 ] ];
gap> a:=AffineSemigroup("equations",[TransposedMat(m),[2,2,2]]);
<Affine semigroup>
gap> at:=GeneratorsOfAffineSemigroup(a);
[ [ 0, 0, 0, 0, 0, 2, 0 ], [ 0, 0, 0, 0, 0, 0, 2 ], [ 0, 0, 0, 0, 2, 0, 0 ], 
  [ 0, 0, 0, 2, 0, 0, 0 ], [ 0, 0, 2, 0, 0, 0, 0 ], [ 0, 2, 0, 0, 0, 0, 0 ], 
  [ 2, 0, 0, 0, 0, 0, 0 ], [ 0, 0, 1, 0, 1, 1, 0 ], [ 0, 0, 1, 1, 0, 0, 1 ], 
  [ 0, 1, 0, 0, 1, 0, 1 ], [ 0, 1, 0, 1, 0, 1, 0 ], [ 1, 0, 0, 0, 0, 1, 1 ], 
  [ 1, 0, 0, 1, 1, 0, 0 ], [ 1, 1, 1, 0, 0, 0, 0 ], [ 0, 0, 0, 1, 1, 1, 1 ], 
  [ 0, 1, 1, 0, 0, 1, 1 ], [ 0, 1, 1, 1, 1, 0, 0 ], [ 1, 0, 1, 0, 1, 0, 1 ], 
  [ 1, 0, 1, 1, 0, 1, 0 ], [ 1, 1, 0, 0, 1, 1, 0 ], [ 1, 1, 0, 1, 0, 0, 1 ] ]
gap> TameDegreeOfAffineSemigroup(a);
4
\end{verbatim}

We use the package \texttt{NormalizInterface} \cite{ni} that utilizes the \texttt{Normaliz} library to compute the Hilbert bases.  Although \texttt{4ti2} \cite{4ti2} can also be used to compute Graver and Hilbert bases via the \texttt{4ti2gap} \cite{4ti2gap} package, in this example \texttt{4ti2} was slower than \texttt{Normaliz}.  The computation of $\mathsf t(\Gamma)$ took 4955 ms on an i7 laptop with 16 GB of memory (\texttt{Normaliz} was compiled without \texttt{OpenMP}, and thus was running in a single thread).  For comparison, an implementation of the procedure described in \cite{c-t} took 1,048,757 ms. 
% 8019 vs 1818478 with 8 threads; but this is cpu time, not real time

The runtime can actually be improved further. According to~\cite{tame}, in order to compute $\mathsf t(\Gamma)$ it suffices to compute $\mathsf t_i(\Gamma)$ with $|\alpha_i| = 4$.  
% ($|\alpha_i|$ is the sum of the coordinates of $\alpha_i$). 
This is accomplished as follows, where \texttt{minimalElementsPrincipalIdealOfFullAffineSemigroup} implements Proposition~\ref{p:principalsaturated}.  

\begin{verbatim}
gap> u:=First(at,x->Sum(x)=4);
[ 0, 0, 0, 1, 1, 1, 1 ]
gap> Mu:=minimalElementsPrincipalIdealOfFullAffineSemigroup(u,a);;
gap> facts:=List(Mu, x->FactorizationsVectorWRTList(x,at));;
gap> Set(facts,TameDegreeOfSetOfFactorizations); 
[ 0, 2, 3, 4 ]
\end{verbatim}

%And so $\mathsf t(\mathbb Z_2^3)=4$ (\cite[Theorem 5.1.2]{tame}). Actually the function \texttt{TameDegreeOfSetOfFactorizations} is not optimized to look just at the tame degree with respect to \texttt{u}.
%
%\begin{verbatim}
%minimalElementsPrincipalIdealOfAffineSemigroup:=function(v,a)
%    local mat, cone, n, hom, par, tot, le, ls, one;
%
%    le:=function(a,b)  #ordinary partial order
%    	return ForAll(b-a,x-> x>=0);
%    end;
%
%    ls:=GeneratorsOfAffineSemigroup(a);
%    n:=Length(ls);
%    one:=[List([1..n],_->1)];
%    mat:=TransposedMat(Concatenation(ls,[-v]));
%    cone:=NmzCone(["inhom_inequalities",mat,"signs",one]);
%    NmzCompute(cone,"DualMode"); 	
%    par:=Set(NmzModuleGenerators(cone), f->f{[1..n]});
%    return List(par,x->x*ls);
%end;
%\end{verbatim}
%This example has been performed with the \texttt{numericalsgps} (\cite{numericalsgps}) package with the use of \texttt{NormalizInterface} (\cite{ni}: a \texttt{GAP}, \cite{gap}, interface to \texttt{Normaliz}, \cite{normaliz}).

It took 2 ms to compute \texttt{at} (the set of atoms of $\Gamma$), and 12 ms to compute \texttt{Mu}.  It took another 420 ms to complete the third line above, and the final line finished in 425 ms.  In total, it takes less than a second to compute the tame degree of $\Gamma$. 
\end{example}

\begin{example}\label{e:tameb2x3}
Running Algorithm~\ref{a:affinetame} for $\Gamma = \mathcal B(\mathbb Z_2\times \mathbb Z_3)$ yields
% \[
% \begin{array}{l@{}l}
% \{ & ( 0, 0, 0, 1, 1 ), ( 0, 0, 2, 0, 0 ), ( 1, 1, 0, 0, 0 ), ( 1, 0, 0, 2, 0 ), 
%      ( 3, 0, 0, 0, 0 ), ( 1, 0, 1, 0, 1 ),  ( 0, 1, 0, 0, 2 ), \\
%    & ( 0, 1, 1, 1, 0 ), ( 0, 3, 0, 0, 0 ), ( 2, 0, 0, 0, 2 ), ( 0, 0, 1, 3, 0 ), 
%      ( 2, 0, 1, 1, 0 ),  ( 0, 0, 1, 0, 3 ), ( 0, 2, 0, 2, 0 ), \\
%    & ( 0, 2, 1, 0, 1 ), ( 1, 0, 0, 0, 4 ), ( 0, 1, 0, 4, 0 ), ( 0, 0, 0, 6, 0 ), 
%      ( 0, 0, 0, 0, 6 ) \}.
% \end{array}
% \]
$\mathsf t(\Gamma) = 8$ after 19214 ms; execution of the procedure presented in \cite{c-t} was killed after several hours.
\end{example}

\begin{example}\label{e:tameb2x2x2x2}
% Now that we have two examples where the process works fine, let us simply modify the first a little bit to see the weakness of our method.
As the size of $G$ increases, the number of atoms (i.e.\ the ambient dimension of $\mathcal B(G)$) quickly makes use of Algorithm~\ref{a:affinetame} infeasible.  Indeed, the block monoid $\mathcal B(\mathbb Z_2^4)$ has 323 atoms, and $\mathcal B(\mathbb Z_2^5)$ has 20367 atoms. Although Algorithm~\ref{a:affinetame} is considerably faster than the general-purpose algorithm introduced in \cite{c-t}, it is still not sufficient to compute $\mathsf t(\mathcal B(\mathbb Z_2^4))$.  
% with an i7 laptop in a reasonable amount of time. 

Given below is the analysis for $\mathsf t(\mathcal B(\mathbb Z_2^4))$ by using the supercomputer \url{alhambra.ugr.es}.  
% (just some cores for the example as we show next).

\begin{enumerate}[(1)]
\item First attempt: 16 cores with 30 GB of internal memory.  
\begin{center}
	\begin{tabular}{ccc}
		\includegraphics[width=\textwidth/4]{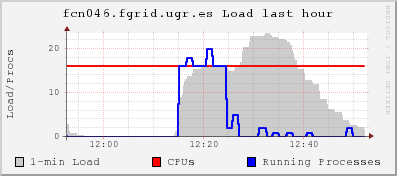} & \includegraphics[width=\textwidth/4]{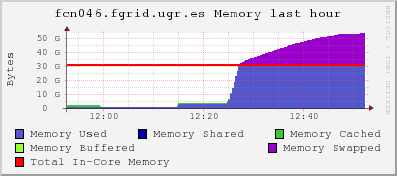} &
		\includegraphics[width=\textwidth/4]{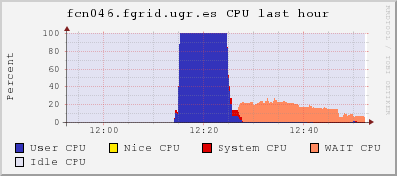} 
	\end{tabular}
\end{center}
Internal memory is quickly exhausted and the system begins using swap memory, decaying the number of processors and the user CPU usage. The process is killed before it finishes.

\item Second attempt: 32 cores with 256 GB of internal memory.  
\begin{center}
	\begin{tabular}{ccc}
		\includegraphics[width=\textwidth/4]{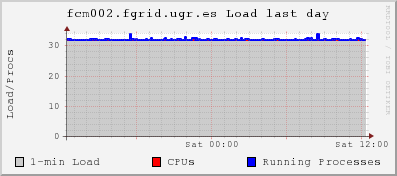} & \includegraphics[width=\textwidth/4]{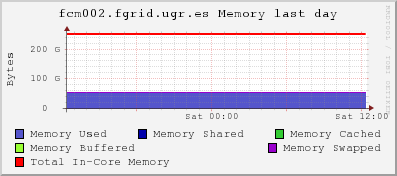} &
		\includegraphics[width=\textwidth/4]{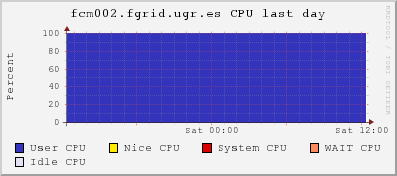} 
	\end{tabular}
\end{center}
% The process was launched on a 72 hour queue.  
After 72 hours of execution, $\texttt{Mu}$ is still not computed.  We contacted W.~Bruns, B.~Ichim and C.~S\"oger (developers of \texttt{Normaliz}) and were told (i) the input has too many variables and few equations to finish in a reasonable time, and (ii) the cone has tens of millions of extremal rays and thus requires extensive memory ($>50$ GB) just to store.  
\end{enumerate}

The point of this example is that even though we may consider $\mathbb Z_2^4$ a ``small'' group, the monoid $\mathcal{B}(\mathbb Z_2^4)$ has 20367 atoms, and this translates to working in an ambient space of dimension 323! Thus, Algorithm~\ref{a:affinetame} will fail to compute the tame degree for this monoid with the computational tools at hand.
\end{example}

% \section*{Conclusions}

% We have given a new interpretation of the tame degree which can be used as a tool for computing it. We have shown that this new machinery allows a nice speed up with respect to the existing one for the calculation of tame degrees. However, we are still far from having the definitive tool, since as seen above, even for simple groups, the number of variables can be too large to be afforded by this method.

% %%%%%%%%%%%%%%%%%%%%%%%%%%%%%%%%%%%%%%%%%%%%%%%%%%%%%%%%%%%%%%%%%%%%%%%%%
\section{Acknowledgements}%%%%%%%%%%%%%%%%%%%%%%%%%%%%%%%%%%%%%%%%%%%%%%%
% %raggedbottom%%%%%%%%%%%%%%%%%%%%%%%%%%%%%%%%%%%%%%%%%%%%%%%%%%%%%%%%%%%%

Supported by FQM-343, FQM-5849 and FEDER funds.  The authors would like to thank Winfried Bruns, Alfred Geroldinger, Alfredo S\'anchez-R.-Navarro,  Richard Sieg and Christof S\"oger  for their comments and discussions on Section~\ref{sec:tamedegree}.  Additionally, we thank Winfried Bruns for pointing out that we only need a single Hilbert basis computation in Algorithm~\ref{a:affinedeltahilbert}, as well as for his assistance with several computations in \texttt{Normaliz}.  The first author also thanks the Centro de Servicios de Inform\'atica y Redes de Comunicaciones (CSIRC), Universidad de Granada, for providing the computing time, especially Rafael Arco Arredondo for his personal time advising which queues should be used in the experiments, and for installing everything needed on the supercomputer \url{alhambra.ugr.es}.

%%%%%%%%%%%%%%%%%%%%%%%%%%%%%%%%%%%%%%%%%%%%%%%%%%%%%%%%%%%%%%%%%%%%%%%%%
%%%%%%%%%%%%%%%%%%%%%%%%%%%%%%%%%%%%%%%%%%%%%%%%%%%%
%%%%%%%%%%%%%%%%%%%%%%%%%%%%%%%%%%%%%%%%%%%%%%%%%%%%%%%%%%%%%%%%%%%%%%%%%

%%%%%%%%%%%%%%%%%%%%%%%%%%%%%%%%%%%%%%%%%%%%%%%%%%%%%%%%%%%%%%%%%%%%%%%%%
\end{document}